\documentclass[1pt,a4paper]{amsart}
\usepackage{a4wide,amsmath,amsfonts,amssymb,amsthm,mathrsfs,setspace,bm,amsxtra,mathtools,wasysym,textcomp,color} 
\mathtoolsset{showonlyrefs}
\usepackage[english]{babel}
\usepackage[all]{xy}
\usepackage{verbatim}

\newtheorem{thm}{Theorem}[section]
\newtheorem{prop}[thm]{Proposition}
\newtheorem{lemma}[thm]{Lemma}
\newtheorem{defin}[thm]{Definition}
\newtheorem{cor}[thm]{Corollary}
\numberwithin{equation}{section}
\theoremstyle{definition}
\newtheorem{rem}[thm]{Remark}
\newenvironment{remark}{\begin{rem}\rm}{\qee\end{rem}}

\newcommand{\qee}{\mbox{\hspace{0.2mm}}\hfill$\triangle$}
\newcommand{\coker}{\operatorname{coker}}
\newcommand{\Aut}{\operatorname{Aut}}
\newcommand{\rk}{\operatorname{rk}}

\newcommand{\Hom}{\operatorname{Hom}}
\newcommand{\End}{\operatorname{End}}

\newcommand{\GL}{\operatorname{GL}}
\newcommand{\im}{\operatorname{Im}}
\newcommand{\Pic}{\operatorname{Pic}}
\newcommand{\Ol}{\mathcal{O}}

\newcommand{\M}{\mathcal{M}}
\newcommand{\E}{\mathcal{E}}
\newcommand{\U}{\mathcal{U}}
\newcommand{\V}{\mathcal{V}}
\newcommand{\W}{\mathcal{W}}

\newcommand{\T}{\mathcal{T}}

\newcommand{\Spec}{\operatorname{Spec}}

\newcommand{\Pk}{P_{\vec{k}}}
\newcommand{\Lk}{L_{\vec{k}}}
\newcommand{\Gk}{G_{\vec{k}}}
\newcommand{\Z}{\mathbb{Z}}

\newcommand{\Com}{\mathbb{C}}
\newcommand{\Pu}{\mathbb{P}^1}
\newcommand{\On}{\Ol_{\Sigma_n}}

\newcommand{\Uk}{\U_{\vec{k}}}
\newcommand{\Vk}{\V_{\vec{k}}}
\newcommand{\Wk}{\W_{\vec{k}}}

\newcommand{\li}{\ell_\infty}

\newcommand{\lra}{\longrightarrow}

\newcommand{\Stab}{\operatorname{Stab}}

\newcommand{\Hilb}{\operatorname{Hilb}}

\title[Hilbert schemes of points of   $\Ol_{\Pu}(-n)$ as quiver varieties]{\large Hilbert schemes of points of  $\Ol_{\Pu}(-n)$\\[10pt]as quiver varieties} 
\subjclass[2010]{14D20;  14D21; 14J60; 16G20} 
\keywords{Hilbert schemes of points, quiver varieties,  quiver representations, special McKay correspondence, Wemyss's reconstruction algebra, ADHM data, monads}

\makeindex

\begin{document}

\maketitle

\vfill

\begin{center}{\sc Claudio Bartocci,$^\P$ 
Ugo Bruzzo,$^{\S\ddag}$\footnote{On leave from Scuola Internazionale Superiore di Studi Avanzati (SISSA), Trieste, Italia}\ 
\\ Valeriano Lanza$^\P$\footnote{Presently at 
IMECC - UNICAMP, Departamento de Matem\'atica, Rua S\'ergio Buarque de Holanda, 651, 13083-970 Campinas-SP, Brazil}
 and Claudio L. S. Rava$^\P$}\footnote{E-mail \tt bartocci@dima.unige.it, bruzzo@sissa.it, valeriano@ime.unicamp.br, clsrava@gmail.com}
  \\[10pt]  \small 
$^\P$Dipartimento di Matematica, Universit\`a di Genova, \\ Via Dodecaneso 35, 16146 Genova, Italia
\\[3pt]
  $^{\S}$ Departamento de Matem\'atica, Universidade Federal de Santa Catarina, \\ 
  Campus Universit\'ario Trindade,
CEP 88.040-900 Florian\'opolis-SC, Brasil \\[3pt]
  $^\ddag$Istituto Nazionale di Fisica Nucleare, Sezione di Trieste
\end{center}

\vfill

\begin{quote}\small {\sc Abstract.} In a previous paper, a realization of the moduli space of framed 
torsion-free sheaves on Hirzebruch surfaces in terms of monads was given. We build upon that result to
construct ADHM data for the Hilbert scheme of points of the total space of the line bundles $\mathcal O(-n)$ on $\mathbb P^1$,
for $n \ge 1$, i.e., the resolutions of the singularities of type
$\frac1n(1,1)$. Basically by implementing a version of the special McKay correspondence, this  ADHM description is in turn used to realize these Hilbert schemes  as irreducible connected components of quiver varieties.
We obtain in this way new examples of quiver varieties which are not of the Nakajima type.
 \end{quote}
{\footnotesize
\setcounter{tocdepth}{1}
\tableofcontents
}

\vfill

\setcounter{tocdepth}{1}
{\small\tableofcontents}

\newpage
\section{Introduction}

If $X$ is a smooth, quasi-projective  surface over $\Com$, the Hilbert scheme of points $\operatorname{Hilb}^c(X)$, which parameterizes
 $0$-dimensional subschemes of $X$ of length $c$, is quasi-projective \cite{groth61} and smooth of dimension $2c$ \cite{fog68}.  In this paper
we study the Hilbert schemes of points of the total spaces of the line bundles $\Ol_{\Pu}(-n)$
(which admit the Hirzebruch surfaces $\Sigma_n$ as projective compactifications). 
These spaces are resolutions of the singularities of type
$\frac1n(1,1)$. 
Our main results
are a description of these Hilbert schemes in terms of  ADHM data, and a consequence, their realization as  irreducible connected components of the moduli spaces of representations of   suitable quivers. {It is worth pointing out that the representation varieties we obtain in this way are not Nakajima quiver varieties.}

\subsection{Motivations and general background}
Among the many occurrences of Hilbert schemes of points  in algebraic geometry, a remarkable one is the role played by the  Hilbert schemes of points of noncompact (usually hyperk\"ahler) surfaces in geometric representation theory. Examples of such spaces are the Hilbert schemes of points of $\Com^2$, and of the minimal resolutions $X_{n,n-1}$ of toric singularities of type $\frac1n(1,n-1)$, with $n\ge 2$.

One of the main links with geometric representation theory is the description of these Hilbert schemes as Nakajima quiver varieties. This basically requires two steps: to consider a space of complex representations of a suitable quiver with relations, and then to construct a GIT (Geometric Invariant Theory) quotient of this with respect to a suitable semistability condition for the representations, in King's sense  \cite{king}. More precisely, 
let us denote by $\mathcal Q$ the Jordan quiver $A_0^{(1)}$, that is, the quiver with one vertex and one loop,
or the affine Dynkin quiver $A_{n-1}^{(1)}$, for $n\ge 2$; one can frame the quiver by doubling the vertices and adding arrows from the old vertices to the new ones, and double it, by duplicating the arrows, with the new arrows going the opposite direction. The space
$\operatorname{Rep}(\mathcal Q^{\mbox{\tiny fr,\ double}},\vec v, \vec w)$    of representations of the double framed quiver associated with $\mathcal Q$ (see \cite{Gin} and Sections \ref{background} and \ref{quiver} in this paper for notation and precise definitions)
has a natural symplectic structure, which is preserved by an action of the group
$\prod_i\operatorname{GL}(v_i)$, giving rise to a moment map; the previously mentioned space of complex representations is indeed the space of representations of $\mathcal Q^{\mbox{\tiny fr,\ double}}$ satisfying the relations that define the moment map. For suitable $\vec v$ and $\vec w$ the GIT quotient of this space, taken with respect to
a certain semistability parameter, is $\operatorname{Hilb}^c(\Com^2)$ if  $\mathcal Q=A_0^{(1)}$,
and $\operatorname{Hilb}^c(X_{n,n-1})$ if $\mathcal Q=A_{n-1}^{(1)}$. In the latter case, the space of representations coincides with the space of left modules of the framed version of Wemyss's reconstruction algebra associated with the toric singularity \cite{W2011} (Wemyss's reconstruction algebra is a quotient of the path algebra of the relevant McKay quiver).  Schiffmann and Vasserot used this description in terms of quiver varieties to define an action of the K-theoretic (or cohomological) Hall algebra associated with $\operatorname{Rep}((A_0^{(1)})^{\mbox{\tiny double}},\vec v)$ on the equivariant K-theory (or cohomology) of $\operatorname{Hilb}^c(\Com^2)$
\cite{SV2013-I, SV2013-II}. Subsequently, Negut extended this construction to
$\operatorname{Hilb}^c(X_{n,n-1})$ \cite{N2015}. This construction is quite important in geometric representation theory, as it can be used to obtain geometric representations of quantum groups, Yangians, vertex algebras, etc. (see {\em loc.~cit.} and references therein). 

In this paper we provide a quiver variety description of the Hilbert schemes of points of the total spaces $X_n$ of the line bundles $\Ol_{\Pu}(-n)$, with $n \ge 1$. For $n \ge 2$, the surface $X_n$ is the minimal resolution of the toric singularity of type $\frac1n(1,1)$; it admits 
the $n$-th Hirzebruch surface $\Sigma_n$ as a projective compactification.   These spaces have been considered in physics in connection with brane counting for string theory compactifications on ``local'' Calabi-Yau manifolds \cite{AOSV,FMP,GSST,OSV,BSS}, gauge theory on Hirzebruch surfaces and applications to the computation of invariants, such as the Betti numbers of moduli spaces of sheaves on Hirzebruch surfaces \cite{BPT,Mansch,BPSS} (see also \cite[App.~D]{BS} for some mathematical developments), topological strings and Gromov-Witten invariants (see \cite{Sz} for a review).

Thus, we supply a new example of a quiver variety {\em outside the universe of the Nakajima quiver varieties}. 
{(Other examples of quiver varieties that are not of the Nakajima type are for instance the affine Laumon spaces, see e.g.~\cite{negut} and references therein, and those associated with the quivers recently studied by Nakajima himself \cite{nakasaw}.)}
This quiver variety is obtained by regarding $\operatorname{Hilb}^c(X_n)$ as the moduli space of rank one torsion-free sheaves on the projective closure $\Sigma_n$ (Hirzebruch surface) of $X_n$ that are trivial on the compactifying divisor, and following the two above mentioned steps: we consider  a space
$\operatorname{Rep}(B_n^{\mbox{\tiny fr}},\vec v, \vec w)$ of representations of the framed version 
$B_n^{\mbox{\tiny fr}}$ of a quotient $B_n$ of the path algebra of a certain quiver $\mathcal   Q_n$, and realize $\operatorname{Hilb}^c(X_n)$ as a suitable GIT quotient of it. Here $\mathcal Q_n$, for $n\ge 2$, is the special McKay quiver associated with the toric singularity of type  $\frac1n(1,1)$, and $B_n$ is the corresponding Wemyss's reconstruction algebra. Thus we provide a new example of a moduli space of sheaves in a minimal resolution of a toric singularity that can be realized as a quiver variety associated with the corresponding special McKay quiver.

As we already mentioned,  the   Nakajima variety
 corresponding  to any quiver carries a  natural symplectic structure. This is essentially due to the fact that  the quiver variety is concocted from 
the associated double quiver. Now, the quiver we associate with the variety $X_2$ is a double,
so that the quiver space we consider, and   the Hilbert space  $\Hilb^c(X_2)$,
are symplectic holomorphic varieties. On the contrary,   the quivers $\mathcal{Q}_{n}^{\textmd{fr}}$ are not doubles for $n\neq 2$.  A result of Bottacin \cite{Bot} implies, however, that $\Hilb^c(X_n)$ carries, for all $n\geq1$, a Poisson structure whose rank is generically maximal.

\subsection{Contents of the paper}
The ADHM description of a moduli space is usually the 
 starting point of its realization as a quiver variety, in that it is the ADHM description which suggests the quiver to consider. For Hilbert schemes of points this description is available in the case of $\Com^2$
 \cite{Nakabook} and for the multi-blowups of $\mathbb C^2$, as provided by the work of A.~A.\ Henni \cite{henni} specialized to the rank one case. So, in the first part of this paper we construct ADHM data for the Hilbert schemes of points of the total space $X_n$ of the line bundle $\Ol_{\Pu}(-n)$. We
 identify the 
 space $\operatorname{Hilb}^c(X_n)$ with the  moduli space $\M^{n}(1,0,c)$  of framed sheaves on the $n$-th Hirzebruch surface $\Sigma_n$ that have rank $1$, vanishing first Chern class, and second Chern class $c_2 = c$ (the framing is a fixed isomorphism with the trivial rank $1$ bundle on a divisor linearly equivalent to the section of $\Sigma_n\to \mathbb P^1$ of positive self-intersection). By exploiting the description of $\M^{n}(1,0,c)$ in terms of monads given in \cite{bbr}, we prove (Theorem \ref{thm0}) that the moduli space $\M^{n}(1,0,c)$ is isomorphic to the quotient $P^n(c) / \GL(c, \Com)\times \GL(c, \Com)$, where $P^n(c)$ is a quasi-affine variety contained in the linear space $\End(\Com^{c})^{\oplus n+2}\oplus\Hom(\Com^{c},\Com)$.  We show this by using the fact  that the partial quotient $P^n(c) / \GL(c, \Com)$ can be assembled by glueing  $c+1$ open sets, each one isomorphic to the space of ADHM data for  $\operatorname{Hilb^c}(\Com^2)$ (Propositions \ref{thm1} and \ref{proGlue}).

In the second part of this paper we show (Theorem \ref{thm:quiver1}) that these Hilbert schemes are irreducible connected components of quiver varieties, namely, they are  embedded as irreducible connected components into varieties of representations of a quiver naturally associated with the ADHM data describing the Hilbert schemes, for a suitable choice of the stability parameter. 

More precisely, we prove:

\smallskip
\noindent{\bf Theorem.} {\em 
 For every $n,c\geq1$, the variety $\Hilb^c(X_n)$ is isomorphic to an irreducible connected component of the quotient
\[
 \operatorname{Rep}(B_n^{\operatorname{fr}},\vec{v}_c,1)^{ss}_{\vartheta_c}/\!/_{\!\vartheta_c}\,\GL(c,\Com)\times\GL(c,\Com)\,,
\]
where $\vec{v}_c=(c,c)$ and $\vartheta_c=(2c,-2c+1)$.
}

\smallskip
Here $  \operatorname{Rep}(B_n^{\operatorname{fr}},\vec{v}_c,1)^{ss}_{\vartheta_c}$
is a representation space asso\-ciated, as we sketched above, with the framed quiver
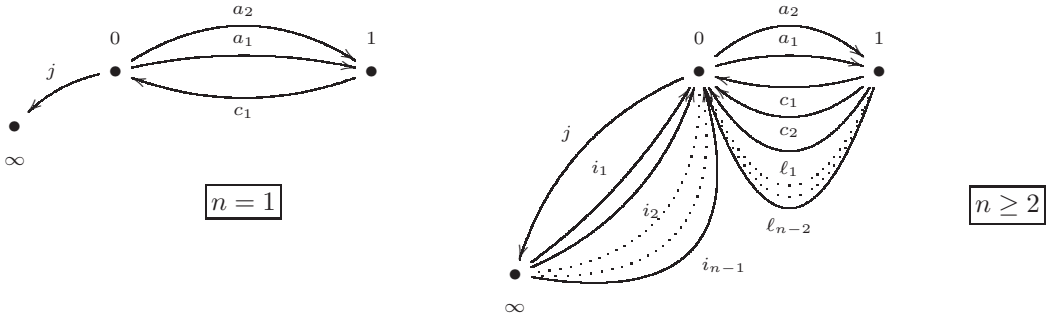
\begin{equation*}
 \xymatrix@R-2.3em{
&\mbox{\scriptsize$0$}&&\mbox{\scriptsize$1$}\\
&\bullet\ar@/_1ex/[ldd]_{j}\ar@/^/[rr]^{a_{1}}\ar@/^4ex/[rr]^{a_{2}}&&\bullet\ar@/^10pt/[ll]^{c_{1}} \\ \\
\bullet&&&\\
\mbox{\scriptsize$\infty$}&&&\\ &&\mbox{\framebox[1cm]{\begin{minipage}{1cm}\centering $n=1$\end{minipage}}}
}
\xymatrix@R-2.5em{
}\qquad\qquad
  \xymatrix@R-2.3em{
&&\mbox{\scriptsize$0$}&&\mbox{\scriptsize$1$}\\
&&\bullet\ar@/_3ex/[llddddddd]_{j}\ar@/^/[rr]^{a_{1}}\ar@/^4ex/[rr]^{a_{2}}&&\bullet\ar@/^/[ll]^{c_{1}}
\ar@/^4ex/[ll]^{c_2}\ar@/^7ex/[ll]^{\ell_1}\ar@{..}@/^10ex/[ll]\ar@{..}@/^11ex/[ll]
\ar@/^12ex/[ll]^{\ell_{n-2}}& \\ \\ \\ \\ \\&&&&&\mbox{\framebox[1cm]{\begin{minipage}{1cm}\centering $n\geq2$\end{minipage}}} \\ \\
\bullet\ar@/_/[rruuuuuuu]^{i_1}\ar@/_3ex/[rruuuuuuu]_{i_2}\ar@{..}@/_6ex/[rruuuuuuu]\ar@{..}@/_8ex/[rruuuuuuu]\ar@/_10ex/[rruuuuuuu]_{i_{n-1}}&&&&&\\
\mbox{\scriptsize$\infty$}&&&&
}
\xymatrix@R-2.5em{
}
\end{equation*}
(for $n=2$ there are no $\ell$ arrows).

This result includes the particular case of the Hilbert scheme of points of  $X_2$, which is isomorphic, as a complex variety, to the ALE space $A_1$. Kuznetsov  has provided, from a different point of view, a description of the Hilbert schemes of the ALE spaces $A_k$ as quiver varieties \cite{kuz}. We check indeed (Corollary \ref{cor:quiver}) that for $n=2$ our representation coincides with that of Kuznetsov
for $A_1$. 

Finally, Appendix A is devoted to proving the rather technical Proposition \ref{pro1}. 

\subsection{Further developments} {Among the many possible developments of the constructions described in this paper,
 one is the study of the chamber structure for the stability parameter used to define the quiver variety. More generally, the results in this paper should  open the way to a number of interesting questions in geometric representation theory, such as the existence of a K-theoretic (cohomological) Hall algebra associated with $\operatorname{Rep}(B_n,\vec v)$, with an action of this algebra on the equivariant K-theory (cohomology)
of $\operatorname{Hilb}^c(X_n)$. On another line,  an interesting and challenging problem is the characterization of the Poisson structure 
of these Hilbert schemes in purely quiver-theoretic terms, perhaps by generalizing the approaches in  \cite{Bie,VDB}.

 \subsection*{Acknowledgments} We thank the referee for the careful reading of the manuscript and for very useful suggestions, which have led to a substantial improvement of the exposition. We also thank Alberto Tacchella for valuable advice. 
 U.B.'s stay   at UFSC was supported by the grant 310002/2015-0 from ``Conselho Nacional de Desenvolvimento Cient\'ifico e Tecnol\'ogico'' (CNPq), Brazil.  He thanks the Algebra and Geometry group at USFC for their hospitality. V.L.~is partially supported by the FAPESP post-doctoral grant 2015/07766-4. Moreover, this work was partially supported by PRIN ``Geo\-metria delle variet\`a algebriche,"  by the University of Genoa's project ``Aspetti matematici della teoria dei campi interagenti e quantizzazione per deformazione," and by GNSAGA-INdAM. U.B. is a member of the VBAC group.
 
 \bigskip\section{Background material} \label{background}
The construction of the ADHM data  is based on the description of the moduli spaces of framed sheaves on the Hirzebruch surfaces $\Sigma_n$ in terms of monads that was  worked out in \cite{bbr}. We briefly review the basic ingredients of that construction.
The $n$-th Hirzebruch surface $\Sigma_n$  is the projective closure of the total space $X_n$ of the line bundle $\mathcal{O}_{\mathbb{P}^1}(-n)$; we shall assume $n > 0$. We denote by $F$ the class in $\Pic(\Sigma_n)$  of the fibre of the natural ruling $\Sigma_n\longrightarrow\Pu$, and by $H$ and $E$ the classes of the sections squaring to $n$ and $-n$, respectively. We shall denote  $\On(p,q) = \On(pH+qF).$
We fix a curve $\ell_{\infty}\simeq\Pu$ in $\Sigma_n$ belonging to the class $H$ and call it  the ``line at infinity.''
A framed sheaf on $\Sigma_n$ is a pair $(\E, \theta)$, where $\E$ is a torsion-free sheaf which 
is trivial along  $\ell_{\infty}$, and $\theta \colon \E\vert_{\ell_{\infty}}\stackrel{\sim}{\longrightarrow}\Ol_{\li}^{\oplus r}$ is an isomorphism, where $r$ is  the rank of $\E$. A morphism between   framed sheaves $(\E, \theta)$, $(\E', \theta')$ is by definition a morphism $\Lambda\colon \E \longrightarrow \E'$ such that  
$\theta'\circ\Lambda\vert_{\ell_{\infty}} = \theta$. The moduli space parameterizing isomorphism classes of framed sheaves $(\E, \theta)$ on $\Sigma_n$ with Chern character $\textrm{ch}(\E) = (r, aE, -c -\frac{1}{2} na^2)$,  where $r, a, c \in \Z$ and $r\geq 1$, will be denoted   $\M^{n}(r,a,c)$. We normalize   the framed sheaves   so that $0\leq a\leq r-1$.

We recall (see e.g. \cite[Definition~II.3.1.1]{Ok}) that a monad $M$ on a scheme $X$ is a three-term  complex of locally free $\mathcal O_X$-modules of finite rank, having nontrivial cohomology only in the middle term:
\begin{equation*}
M\,:\qquad\xymatrix{
0 \ar[r]& \mathcal{U}\ar[r]^a &
\mathcal{V}
\ar[r]^b &\mathcal{W} \ar[r] &0\,.
}
\end{equation*}
The cohomology of the monad 
is a coherent $\mathcal O_X$-module.
 A \emph{morphism} (\emph{isomorphism}) \emph{of monads} is     a morphism (isomorphism) of complexes.

As proved in \cite{bbr},
a framed sheaf 
$(\E, \theta)$  on $\Sigma_n$, having   invariants $(r,a,c)$,  is isomorphic to the cohomology of a monad 
\begin{equation}
\xymatrix{
M(\alpha,\beta):&0 \ar[r] & \Uk \ar[r]^-{\alpha} & \Vk \ar[r]^-{\beta} & \Wk \ar[r] & 0
}\,, \label{fundamentalmonad}
\end{equation}
where $\vec{k}$ denotes the quadruple $ (n,r,a,c)$, and we have set 
\begin{equation} 
\Uk:=\On(0,-1)^{\oplus k_1},\quad
\Vk:=\On(1,-1)^{\oplus k_2} \oplus \On^{\oplus k_4},\quad
\Wk:=\On(1,0)^{\oplus k_3}\,,
\end{equation}
with
\begin{equation} 
 k_1=c+\dfrac{1}{2}na(a-1),\quad
k_2=k_1+na,\quad
k_3=k_1+(n-1)a,\quad
k_4=k_1+r-a\,.
\label{k_i}
\end{equation}

The set $\Lk$ of pairs in $\Hom(\Uk,\Vk)\oplus\Hom(\Vk,\Wk)$ fitting into the complex  \eqref{fundamentalmonad},
such that the cohomology of the complex is torsion-free and trivial at infinity, is a smooth algebraic variety. One can introduce a principal $\GL(r,\Com)$-bundle $\Pk$ over $\Lk$, whose fibre at a point $(\alpha,\beta)$ is   identified with the space of framings for the corresponding cohomology of \eqref{fundamentalmonad}. 
The algebraic group $\Gk=\Aut(\Uk)\times\Aut(\Vk)\times\Aut(\Wk)$ 
acts freely on $\Pk$, and the moduli space $\M^{n}(r,a,c)$ can be described as the quotient $\Pk/\Gk$ \cite[Theorem 3.4]{bbr}. This space is nonempty if and only if $c + \frac{1}{2} na(a-1) \geq 0$, and when nonempty,   is a smooth algebraic variety of dimension $2rc + (r-1) na^2$.

When  $r=1$ we can assume $a=0$, so that the double dual $\E^{\ast\ast}$ of $\E$  is isomorphic to the structure  sheaf $\On$. As a consequence, since $\E$ is trivial on $\ell_{\infty}$,  the mapping carrying $\E$ to the schematic support of $\On/\E$ yields an isomorphism
\begin{equation} 
\M^{n}(1,0,c) \simeq \operatorname{Hilb}^c (\Sigma_n \setminus \ell_{\infty}) = \operatorname{Hilb}^c (X_{n})\,,
\label{eqMH}
\end{equation}
where $X_{n}$ is the total space of the line bundle $\Ol_{\Pu}(-n)$. We shall freely use the isomorphism \eqref{eqMH}  in the rest of the paper.

We also  fix some notation about quiver representations (see \cite{Gin} for details). A quiver $\mathcal{Q}$ is a finite oriented graph, given by a set of vertices $I$ and a set of arrows $E$. The path algebra $\Com\mathcal{Q}$ is the $\Com$-algebra with basis the paths in $\mathcal{Q}$ and with a product given by the concatenation of paths whenever possible, zero otherwise. Usually  one includes among the generators of $\Com\mathcal{Q}$ a complete set of orthogonal idempotents $\{e_{i}\}_{i\in I}$: this can be considered a subset of $E$ by regarding $e_{i}$ as a loop of ``length zero'' starting and ending at the $i$-th vertex. A  (complex) representation of a quiver $\mathcal{Q}$ is a pair $(V,X)$, where $V= \bigoplus_{i\in I}V_{i}$ is an $I$-graded complex vector space and $X=(X_{a})_{a\in E}$ is a collection of linear maps such that $X_{a}\in\Hom_{\Com}(V_{i},V_{j})$ whenever the arrow $a$ starts at the vertex $i$ and terminates at the vertex $j$. We say that a representation $(V,X)$ is supported by $V$, and denote by $\operatorname{Rep}(\mathcal{Q},V)$ the space of representations of $\mathcal{Q}$ supported by 
$V$. Morphisms and direct sums of representations are defined in an obvious way; it can be shown that the abelian category of complex representations of $Q$ is equivalent to the category of left $\Com\mathcal{Q}$-modules. In particular, 
a sub-representation of a given representation $(V,X)$ is a pair $(S,Y)$, where $S$ is an $I$-graded subspace of $V$ which is preserved by the linear maps $X$, and $Y$ is the restriction of $X$ to $S$.  

We  consider only finite-dimensional representations. If $\dim_{\Com} V_{i}=v_{i}$,
a representation $(V,X)$ of $\mathcal{Q}$ is said to be $\vec{v}$-dimensional, where $\vec{v}=(v_{i})_{i\in I}\in\mathbb{N}^{I}$.  With an abuse of notation, after fixing a $\vec{v}$-dimensional vector space $V$, we write $\operatorname{Rep}(\mathcal{Q},\vec{v})$ instead of $\operatorname{Rep}(\mathcal{Q},V)$. {When we consider a framed quiver (see Section \ref{quiver} for this notion), for notational convenience we shall separately denote by $\vec w$ the dimension vector of the vector spaces associated with the framing vertices.}

More generally one can define the representations of a quotient algebra $B=\Com\mathcal{Q}/J$, for some ideal $J$ of the path algebra $\Com\mathcal{Q}$. These are representations $(V,X)$ of $\mathcal{Q}$, whose linear maps $X=(X_{a})_{a\in E}$ satisfy the relations given by the elements of $J$. The abelian category of complex representations of $B$ is equivalent to the category of left $B$-modules.  We denote by $\operatorname{Rep}(B,\vec{v})$ the space of representations of $B$ supported by a given $\vec{v}$-dimensional vector space $V$. 
There is a natural action of $\prod_i\GL(v_i)$ on $\operatorname{Rep}(B,\vec{v})$ given by change of basis. One would like to consider the space of isomorphism classes of $\vec{v}$-dimensional representations of $B$, but unfortunately in most cases this space is   ``badly behaved.'' To overcome this drawback, following A.~King's approach \cite{king}, one introduces a notion of (semi)stability depending on the choice of a parameter $\vartheta\in\mathbb{R}^{I}$, considers the subset $\operatorname{Rep}(B,\vec{v})^{ss}_\vartheta$ of $\operatorname{Rep}(B,\vec{v})$ consisting of semistable representations, and   takes the corresponding GIT quotient $\operatorname{Rep}(B,\vec{v})^{ss}_\vartheta/\!/_{\!\vartheta}\prod_i\GL(v_i)$. 

\bigskip
\section{ADHM data}\label{sectionmain}

In this section we construct ADHM data for the Hilbert scheme of points of the total spaces $X_n$
of the line bundles $\Ol_{\Pu}(-n)$. First we show that the Hilbert schemes can be covered by open subsets, each of which is isomorphic to the Hilbert scheme of $\mathbb C^2$,
and therefore admits Nakajima's ADHM description; then we prove that these ``local data'' can be glued together to provide ADHM data for the Hilbert schemes of $X_n$.

We denote by $P^{n}(c)$ the subset of the vector space $\End(\Com^{c})^{\oplus n+2}\oplus\Hom(\Com^{c},\Com)$ whose elements $\left(A_1,A_2;C_1,\dots,C_{n};e\right)$ satisfy the   conditions
\begin{enumerate}
 \item[(P1)]
\begin{equation*}
\begin{cases}
A_1C_1A_{2}=A_2C_{1}A_{1}&\qquad\text{when $n=1$}\\[15pt]
\begin{aligned}
A_1C_q&=A_2C_{q+1}\\
C_qA_1&=C_{q+1}A_2
\end{aligned}
\qquad\text{for}\quad q=1,\dots,n-1&\qquad\text{when $n>1$;}
\end{cases}
\end{equation*} \smallskip
\item[(P2)]
$A_1+\lambda A_2$ is a \emph{regular pencil} of matrices; equivalently, there exists $[\nu_{1},\nu_{2}]\in\Pu$ such that $\det(\nu_1A_1+\nu_2A_2)\neq0$;\smallskip
\item[(P3)]
for all values of the parameters $\left([\lambda_1,\lambda_2],(\mu_1,\mu_{2})\right)\in\Pu\times\Com^{2}$ such that
\begin{equation*}
\lambda_{1}^{n}\mu_{1}+\lambda_{2}^{n}\mu_{2}=0
\end{equation*}
there is no nonzero vector $v\in\Com^c$ such that
\begin{equation*}
\left\{
\begin{array}{l}
C_{1}A_{2}v=-\mu_1v\\
C_{n}A_{1}v=(-1)^n\mu_2v\\
v\in\ker e
\end{array}\right.
\qquad\text{and}\qquad\left(\lambda_2{A_1}+\lambda_1{A_2}\right)v=0\,.
\end{equation*}\smallskip
\end{enumerate}
The group 
$\GL(c, \Com)\times \GL(c, \Com)$  acts on $P^n(c)$ according to the rule
\begin{equation} (A_i, \, C_j, e ) \mapsto (\phi_2A_i\phi_1^{-1},\, 
\phi_1C_j\phi_2^{-1},\,
e\phi_1^{-1})
\label{eqrho}
\end{equation}
for $i=1,2$, $j=1,\dots,n$, $ (\phi_1,\phi_2)\in\GL(c, \Com)\times \GL(c, \Com)$.
 
 \begin{thm}
\label{thm0}
$P^{n}(c)$ is a principal $\GL(c, \Com)\times \GL(c, \Com)$-bundle over $\operatorname{Hilb}^c (X_{n})$.
\end{thm}

The remainder of this Section is devoted to proving Theorem \ref{thm0}.
At first, we provide an ADHM description for each open set of an open cover of $\operatorname{Hilb}^c (X_{n})$. 
If we fix $c+1$ distinct fibres $F_0,\dots,F_{c}\in F$, for any $[(\E,\theta)]\in \operatorname{Hilb}^c (X_{n})$ there exists at least one $m\in\{0,\dots,c\}$ such that $\E|_{F_{m}}\simeq\Ol_{F_{m}}$. We choose the fibres $F_{m}$
as the closed subvarieties cut in
\begin{equation}
\Sigma_n = \left\{([y_1,y_2],[x_1,x_2,x_3])\in\mathbb{P}^1\times\mathbb{P}^2\;|\;x_1y_1^n=x_2y_2^n \right\}
\label{eqSn}
\end{equation}
by the equations
\begin{equation}
F_m=\{[y_1,y_2]=[c_m,s_m]\},\qquad m=0,\dots,c\,,
\label{eq12}
\end{equation}
where
\begin{equation}
c_m=\cos\left(\pi\frac{m}{c+1}\right), \qquad s_m=\sin\left(\pi\frac{m}{c+1}\right)\,.
\label{eqcmsm}
\end{equation}
We obtain in this way  an open cover $\left\{U^{n,c}_m\right\}_{m=0,\dots,c}$ of $\operatorname{Hilb}^c (X_{n})$ by letting
\begin{equation}\label{U} 
U^{n,c}_m:=\left\{[(\E,\theta)]\in\operatorname{Hilb}^c (X_{n})\,\bigm\vert\,
\E_{\vert F_m} \simeq \Ol_{F_m}
\right\}.
\end{equation}
Each of these spaces is isomorphic to  the Hilbert scheme of points of $\Com^2$, so that it admits Nakajima's ADHM description \cite[Theorem.~1.9]{Nakabook}
in terms of two $c\times c$ matrices $b_1$, $b_2$ and a row $c$-vector $e$,
satisfying the conditions 
\begin{itemize}
\item[(T1)]
$[b_{1},b_{2}]=0\,;
$ \smallskip
\item[(T2)]
for all $(z,w)\in\Com^2$ there is no nonzero vector $v\in\Com^{c}$ such that
\begin{equation*}
\left\{
\begin{array}{l}
b_{1}v=zv\\
b_{2}v=wv\\
v\in\ker e\,.
\end{array}
\right.
\end{equation*}
\end{itemize} 
The space of triples  $( b_1,b_2,e)$ satisfying the previous two conditions will be denoted by $\T(c)$. Elements $\phi$ of the 
 group   $\GL(c, \Com)$ act on $\T(c)$ according to the rule
\begin{equation}
( b_1,b_2,e) \mapsto  (\phi \, b_{1}\, \phi^{-1}, \phi\, b_{2}\, \phi^{-1}, e\,\phi^{-1})\,.
\label{eqGL(c)onK}
\end{equation} 
 {Note that condition (T2) is the so-called \emph{co-stability condition}, while Nakajima in \cite[Theorem.~1.9]{Nakabook} used the \emph{stability condition} (which is satisfied by the transpose matrices $({}^{t}b_{1},{}^{t}b_{2},{}^{t}e)$). In particular, t}his explains the difference between \eqref{eqGL(c)onK} and the $\GL(c, \Com)$-action used by Nakajima.

 The ADHM data for the open set $U^{n,c}_m$ will be denoted by $(b_{1m},b_{2m},e_{m})$; 
 the next Proposition gives 
the transition functions on the intersections.
\begin{prop}
\label{thm1}
The intersections $U^{n,c}_{ml}=U^{n,c }_m\,\cap \,U^{n,c }_l$ are characterized by the conditions
$$
\det\left(c_{m-l}\bm{1}_c-s_{m-l}b_{1m}\right) \neq0 \,,
$$
where $c_{m}$ and $s_{m}$ are the numbers defined in eq.~\eqref{eqcmsm}. On any of these intersections, the ADHM data are related by the equations
\begin{equation*}
\begin{cases}
b_{1l}&=\left(c_{m-l}\bm{1}_c-s_{m-l}b_{1m}\right)^{-1}\left(s_{m-l}\bm{1}_c+c_{m-l}b_{1m}\right)\\
b_{2l}&=\left(c_{m-l}\bm{1}_c-s_{m-l}b_{1m}\right)^{n}b_{2m}\\
e_{l}&=e_{m}\,.
\end{cases}
\end{equation*}
\label{pro1}
\end{prop}
\begin{proof} The proof of this result is given in Appendix \ref{appa}.
\end{proof} 
We introduce the matrices
\begin{equation}
\begin{aligned}
A_{1m}&=c_mA_1-s_mA_2\,,\qquad
A_{2m}=s_mA_1+c_mA_2\,,\\
E_{m}&=\left[\sum_{q=1}^{n}\binom{n-1}{q-1}c_{m}^{n-q}s_{m}^{q-1}C_q\right]A_{2m}\,,
\end{aligned}
\label{eq410}
\end{equation}
where $m=0,\dots,c$. Since the polynomial $\det(\nu_{1}A_1+\nu_{2}A_2)$ has at most $c$ distinct roots in $\Pu$, the $\GL(c, \Com)\times \GL(c, \Com)$-invariant open subsets
\begin{equation}
P^n(c)_m=\left\{\left(A_1,A_2;C_{1},\dots,C_{n};e\right)\in P^n(c)\left|
\ \det A_{2m}\neq0
\right.\right\}\,,\qquad m=0,\dots,c\,,
\label{pcnm}
\end{equation}
cover $P^{n}(c)$.
If  we also define  the matrices
$B_{m}=A_{2m}^{-1}A_{1m}$,
the linear data $(B_{m},E_{m},e;A_{2m})$ provide local affine coordinates for $P^n(c)$.
\begin{prop}
\label{lmTomega}
The morphism
\begin{equation*}
\begin{array}{rccl}
\zeta_{m}\colon&P^{n}(c)_{m}&\longrightarrow&\left[\End(\Com^{c})^{\oplus 2}\oplus\Hom(\Com^{c},\Com)\right]\times\GL(c, \Com)\\
&(A_{1},A_{2};C_{1},\dots,C_{n};e)&\longmapsto&
\left(B_{m},E_{m},e;A_{2m}\right)
\end{array}
\end{equation*}
is an isomorphism onto $\T(c)\times\GL(c, \Com)$. The induced $\GL(c, \Com)\times \GL(c, \Com)$-action is given by
\begin{equation}
(B_{m},E_{m},e;A_{2m}) \mapsto (\phi_{1}B_{m}\phi_{1}^{-1},\,
\phi_{1}E_{m}\phi_{1}^{-1},\,e\phi_{1}^{-1};\,\phi_{2}A_{2m}\phi_{1}^{-1})\,.
\label{eqrhol}
\end{equation}
\end{prop}
We divide the proof of Proposition \ref{lmTomega} into a few steps. First we define the matrices $\sigma^h_{m}=(\sigma^{h}_{m;pq})_{0\leq p,q \leq h}$ for all $h\geq0$ and $m\in\Z$ by means of the equations
\begin{equation}
(s_{m}\mu_{1}+c_{m}\mu_{2})^{p}(c_{m}\mu_{1}-s_{m}\mu_{2})^{h-p}=\sum_{q=0}^{h}\sigma^{h}_{m;pq}\mu_{2}^{q}\mu_{1}^{h-q}
\label{eqsigma}
\end{equation}
for any $(\mu_{1},\mu_{2})\in\Com^{2}$ and $p=0,\dots,h$. Note that $\sigma^h_{m}\sigma^h_{l}=\sigma^h_{m+l}$ and $\sigma^h_{0}=\bm{1}_{h+1}$. In particular, $\sigma^h_{m}$  is invertible for all $h\geq0$ and $m\in\Z$.

\begin{lemma}
\label{propSyst}
Assume   $n>1$. If the matrices $A_{1},A_{2}\in\End(\Com^{c})$ satisfy condition \mbox{\rm(P2)},  then the system
$A_1C_q=A_2C_{q+1}$, $q=1,\dots,n-1$,
with $C_{q}\in\End(\Com^{c})$,  has      maximal rank, namely, $(n-1)c^{2}$. In particular, if $\det A_{2m}\neq0$, the general solution of 
the previous linear system is
\begin{equation}
\begin{pmatrix}
C_{1}\\
\vdots\\
\vdots\\
C_{n}
\end{pmatrix}=(\sigma^{n-1}_{m}\otimes\bm{1}_{c})
\begin{pmatrix}
\bm{1}_{c}\\
B_{m}\\
\vdots\\
B_{m}^{n-1}
\end{pmatrix}
D_m\,,
\label{eq64}
\end{equation}
where we have chosen as free parameter the  matrix
\begin{equation}
D_m=
\sum_{q=1}^{n}\binom{n-1}{q-1}c_{m}^{n-q}s_{m}^{q-1}C_q\,.
\label{eq43}
\end{equation}
\end{lemma}
\begin{proof}
The maximality of the rank of the system follows from condition (P2) by arguing as in \cite[pp.~29-30]{Gan}. Eq. \eqref{eq64} can be  verified by direct substitution.
\end{proof}
Since $E_{m}=D_{m}A_{2m}$, the morphism $\zeta_{m}$ is injective.

Next we prove that $\im\zeta_{m}\subseteq\T(c)\times\GL(c, \Com)$. This follows from the next  Lemma.

\begin{lemma}
(i) For all $(B_{m},E_{m},e;A_{2m})\in\im\zeta_{m}$,  one has
$
[B_{m},E_{m}]=0
$.

(ii)  Let $(A_{1},A_{2};C_{1},\dots,C_{n};e)\in\End(\Com^{c})^{\oplus(n+2)}\oplus\Hom(\Com^{c},\Com)$ be an $(n+3)$-tuple such that condition \mbox{\rm(P1)} is satisfied and $\det A_{2m}\neq0$. Then
 \begin{itemize}
 \item 
if $[\lambda_1,\lambda_2]=[c_m,s_m]$, condition \mbox{\rm(P3)} is trivially satisfied; 
\item
if $[\lambda_1,\lambda_2]\neq[c_m,s_m]$, condition \mbox{\rm(P3)} holds  if and only if condition \mbox{\rm(T2)} holds for the triple $(B_{m},E_{m},e)$.
\end{itemize}
\label{lmP3}
\end{lemma}
\begin{proof}
(i) 
For all $n\geq1$, condition (P1) implies  
\begin{equation}
0=A_1C_qA_2-A_2C_qA_1=A_{1m}C_{q}A_{2m}-A_{2m}C_{q}A_{1m}=-A_{2m}[C_qA_{2m} , B_m]
\end{equation}
for $q=1,\dots,n$ and   $m=0,\dots,c$. The thesis follows from eq.~\eqref{eq410}.

(ii) 
If $[\lambda_1,\lambda_2]=[c_m,s_m]$ one has $\lambda_{2}A_{1}+\lambda_{1}A_{2}=\lambda A_{2m}$ for some $\lambda\in\Com^{*}$. This proves the first statement.

Assume that $[\lambda_1,\lambda_2]\neq[c_m,s_m]$. One has
\begin{equation}
\begin{gathered}
\lambda_{2}A_{1}+\lambda_{1}A_{2}=\lambda A_{2m}(B_{m}-z\bm{1}_{c})\qquad\text{where}\qquad z=\frac{c_{m}\lambda_{1}+s_{m}\lambda_{2}}{s_{m}\lambda_{1}-c_{m}\lambda_{2}}\\[5pt]
\lambda_{1}^{n}\mu_{1}+\lambda_{2}^{n}\mu_{2}=0\qquad\text{if and only if}\qquad
\left\{
\begin{aligned}
\mu_{1}&=(-1)^{n-1}(s_{m}z-c_{m})^{n}w\\
\mu_{2}&=(-1)^{n}(c_{m}z+s_{m})^{n}w
\end{aligned}
\right.
\end{gathered}
\label{eqAmu}
\end{equation}
for some $\lambda\in\Com^{*}$ and $w\in\Com$. By using eqs. \eqref{eq64} and \eqref{eqAmu} one can prove  the equivalence between the   systems
\begin{equation}
\left\{
\begin{aligned}
(\lambda_{2}A_{1}+\lambda_{1}A_{2})v&=0\\
C_{1}A_{2}v&=-\mu_{1}v\\
C_{n}A_{1}v&=(-1)^{n}\mu_{2}v
\end{aligned}
\right.\qquad\mbox{and}\qquad
\left\{
\begin{aligned}
(B_{m}-z\bm{1}_{c})v&=0\\
(s_{m}z-c_{m})^{n}(E_{m}-w\bm{1}_{c})v&=0\\
(c_{m}z+s_{m})^{n}(E_{m}-w\bm{1}_{c})v&=0
\end{aligned}
\right. .
\label{eqEquiv}
\end{equation}
The thesis follows   as the polynomials $s_{m}z-c_{m}$ and $c_{m}z+s_{m}$ are coprime in $\Com[z]$.
\end{proof}
The following result will be useful in the next section.
\begin{cor}\label{useful}
Let $(A_{1},A_{2};C_{1},\dots,C_{n})\in\End(\Com^{c})^{\oplus(n+2)}$ be an $(n+2)$-tuple such that conditions \mbox{\rm(P1)} and \mbox{\rm(P2)} are satisfied. There exists a nonzero vector $v\in\Com^{c}$ and parameters $\left([\lambda_{1},\lambda_{2}],(\mu_{1},\mu_{2})\right)\in\Pu\times\Com^{2}$ such that
\begin{equation}
\lambda_{1}^{n}\mu_{1}+\lambda_{2}^{n}\mu_{2}=0
\label{eqlambdamu}
\end{equation}
and
\begin{equation}
\left\{
\begin{aligned}
(\lambda_{2}A_{1}+\lambda_{1}A_{2})v&=0\\
C_{1}A_{2}v&=-\mu_{1}v\\
C_{n}A_{1}v&=(-1)^{n}\mu_{2}v
\end{aligned}
\right.
\label{eqPP3}
\end{equation}
\end{cor}
\begin{proof}
By condition (P2) there exists $m\in\{0,\dots,c\}$ such that $\det A_{2m}\neq 0$, so that one can define $B_{m}$. One can also define $E_{m}$
according to eq.~\ref{eq410}.
 By arguing as in the proof of Lemma \ref{lmP3}, one deduces that $B_{m}$ and $E_{m}$ commute, hence they have a common eigenvector $v\in\Com^{c}$ with eigenvalues, say, $z$ and $w$, respectively. By eq. 
\eqref{eqEquiv} this is equivalent to eq. \eqref{eqPP3}, provided   we define
\begin{align}
\lambda_1&=c_mz+s_m\,, &\qquad \lambda_2&=s_mz-c_m\,,\\
\mu_{1}&=(-1)^{n-1}(s_{m}z-c_{m})^{n}w\,, & \qquad
\mu_{2}&=(-1)^{n}(c_{m}z+s_{m})^{n}w\,.
\end{align}
Eq. \eqref{eqlambdamu} can now be checked.
\end{proof}

Finally, we prove that $\T(c)\times\GL(c, \Com)\subseteq\im\zeta_{m}$. Let $(b_{1},b_{2},e;A)\in \T(c)\times\GL(c, \Com)$; if  
\begin{equation*}
A_{1}=A(c_{m}b_{1}+s_{m}\bm{1}_{c})\,,\quad
A_{2}=A(-s_{m}b_{1}+c_{m}\bm{1}_{c})\,,
\end{equation*}
\begin{equation}
\begin{pmatrix}
C_{1}\\
\vdots\\
\vdots\\
C_{n}
\end{pmatrix} =(\sigma^{n-1}_{m}\otimes\bm{1}_{c})
\begin{pmatrix}
\bm{1}_{c}\\
b_{1}\\
\vdots\\
b_{1}^{n-1}
\end{pmatrix}
b_{2}A^{-1}\,,
\label{eqACe}
\end{equation}
then $(A_1,A_2;C_1,\dots,C_n;e)\in P^n(c)_m$ and $\zeta_{m}(A_1,A_2;C_1,\dots,C_n;e)=(b_{1},b_{2},e;A)$. One can verify by substitution that condition (P1) holds. Note now that by substituting \eqref{eqACe} into eq.~\eqref{eq410} one gets
\begin{equation*}
A_{1m}=Ab_{1},\qquad A_{2m}=A, \qquad E_{m}=b_{2}\,.
\end{equation*}
This shows that $A_{2m}$ is invertible, and in particular, condition (P2) holds. Since $B_{m}=b_{1}$, by Lemma \ref{lmP3} condition (P3) holds as a consequence of (T2). This concludes the proof of Proposition \ref{lmTomega}.

We now compute the transition functions on the intersections $P^{n}(c)_{ml}=P^{n}(c)_{m}\cap P^{n}(c)_{l}$, for $m,l=0,\dots,c$. First observe that 
\begin{equation*}
\zeta_{m}\left(P^{n}(c)_{ml}\right)=\T(c)_{m,l}\times\GL(c, \Com)
\end{equation*}
as a consequence of the identity
\begin{equation}
\begin{split}
A_{2l}&=
\begin{pmatrix}
s_{l}\bm{1}_{c} & c_{l}\bm{1}_{c}
\end{pmatrix}
\begin{pmatrix}
c_{m}\bm{1}_{c} & s_{m}\bm{1}_{c}\\
-s_{m}\bm{1}_{c} & c_{m}\bm{1}_{c}
\end{pmatrix}
\begin{pmatrix}
A_{1m}\\
A_{2m}
\end{pmatrix}
=A_{2m}(c_{m-l}\bm{1}_{c}-s_{m-l}B_{m})
\end{split}
\label{eqA2l}
\end{equation}
(the notation $\T(c)_{m,l}$ in introduced in eq.\ \eqref{tcml} of Appendix \ref{appa}).
\begin{prop}
\label{proGlue}
One has a commutative triangle
\begin{equation}
\xymatrix{
& P^{n}(c)_{ml} \ar[ld]_{\zeta_{m,l}} \ar[rd]^{\zeta_{l,m}}\\
\T(c)_{m,l}\times\GL(c, \Com) \ar[rr]^-{\omega_{lm}} & & \T(c)_{l,m}\times\GL(c, \Com)\,,
}
\label{eqtri}
\end{equation}
where $\zeta_{m,l}$ and $\zeta_{l,m}$ are the restrictions of $\zeta_{m}$ and $\zeta_{l}$, respectively, and
\begin{equation*}
\omega_{lm}(B_{m},E_{m},e;A_{2m})=\left({\varphi}_{lm}(B_{m},E_{m},e),A_{2m}(c_{m-l}\bm{1}_{c}-s_{m-l}B_{m})\right)\,,
\end{equation*}
the functions ${\varphi}_{lm}$ being defined analogously to the transition functions in Proposition \ref{pro1}. The transition functions $\omega_{lm}$ are $\GL(c,\Com)\times \GL(c, \Com)$-equivariant.
\end{prop}
\begin{proof}
We want to express $(B_l,E_l,e;A_{2l})$ in terms of $(B_m,E_m,e;A_{2m})$. We already have eq.~\eqref{eqA2l}; analogously, one can prove $A_{1l}=A_{2m}(s_{m-l}\bm{1}_{c}+c_{m-l}B_{m})$. It follows that $B_{l}=(c_{m-l}\bm{1}_{c}-s_{m-l}B_{m})^{-1}(s_{m-l}\bm{1}_{c}+c_{m-l}B_{m})$.
As for $E_{l}$, one has
\begin{equation*}
E_{l}=\left[\sum_{p=1}^{n}\sigma^{n-1}_{-l;0,p-1} C_{p}\right]A_{2l}
=\left[\sum_{p=0}^{n-1} \sigma^{n-1}_{m-l;0p}B_{m}^{p}\right]E_{m}A_{2m}^{-1}A_{2l}
=(c_{l-m}\bm{1}_{c}-s_{l-m}B_{m})^{n}E_{m}\,,
\end{equation*}
where we have used eq. \eqref{eq64}, the relation $\sigma^{n-1}_{m-l}=\sigma^{n-1}_{-l}\sigma^{n-1}_{m}$, and Lemma \ref{lmP3}.

The equivariance of $\omega_{lm}$ is straightforward, and this completes the proof.
\end{proof}

From  Proposition \ref{lmTomega} we have \begin{equation}
P^{n}(c)_{m}/\GL(c, \Com)\times \GL(c, \Com)\simeq \T(c)/\GL(c, \Com) \simeq  U^{n,c}_m\,;
\label{eqDelta}
\end{equation}
moreover, there is an equivariant isomorphism $P^{n}(c)_{m}\simeq \T(c) \times \GL(c, \Com)$. As
  $\T(c)$ is a  principal $\GL(c, \Com)$-bundle over $\T(c)/\GL(c, \Com)$, the space $P^{n}(c)_{m}$ turns out to be a principal $\GL(c, \Com)\times \GL(c, \Com)$-bundle over $U^{n,c}_{m}$. 
Propositions \ref{pro1} and \ref{proGlue} now imply  that $P^n(c)$ is a principal $\GL(c, \Com)\times \GL(c, \Com)$-bundle, and this completes the proof of Theorem \ref{thm0}.

\bigskip
\section{Hilbert schemes as quiver varieties}\label{quiver}
In this   section we   prove that the Hilbert schemes of points of the total space $X_n$  of the line bundles $\Ol_{\Pu}(-n)$ are isomorphic to suitable moduli spaces of quiver representations. 

\subsection{Quiver varieties}
Given two arrows $x,y$, we adopt the notation $xy$ for ``\emph{$y$ followed by $x$}.''
We introduce the quivers $\mathcal{Q}_{1}$ and $\mathcal{Q}_{2}$ 
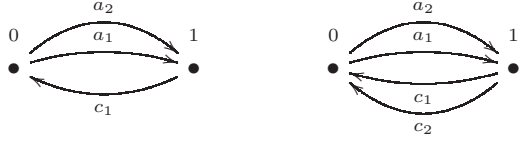
\begin{equation}
\xymatrix@R-2.3em{
\mbox{\scriptsize$0$}&&\mbox{\scriptsize$1$}\\
\bullet\ar@/^/[rr]^{a_{1}}\ar@/^4ex/[rr]^{a_{2}}&&\bullet\ar@/^10pt/[ll]^{c_{1}}
}
\qquad\qquad
\xymatrix@R-2.3em{
\mbox{\scriptsize$0$}&&\mbox{\scriptsize$1$}\\
\bullet\ar@/^/[rr]^{a_{1}}\ar@/^4ex/[rr]^{a_{2}}&&\bullet\ar@/^/[ll]^{c_{1}}
\ar@/^4ex/[ll]^{c_2}
}
\end{equation}
and, for $n\geq3$, the quivers  $\mathcal{Q}_{n}$ 
\begin{equation}
\xymatrix@R-2.3em{
\mbox{\scriptsize$0$}&&\mbox{\scriptsize$1$}\\
\bullet\ar@/^/[rr]^{a_{1}}\ar@/^4ex/[rr]^{a_{2}}&&\bullet\ar@/^/[ll]^{c_{1}}
\ar@/^4ex/[ll]^{c_2}\ar@/^7ex/[ll]^{\ell_1}\ar@{..}@/^10ex/[ll]\ar@{..}@/^11ex/[ll]
\ar@/^12ex/[ll]^{\ell_{n-2}}
} 
\label{eqQui}
\end{equation}

 {
Consider the relations}
\begin{align}
&a_1c_1a_2=a_2c_1a_1\,, \label{In1}\\
 &a_1c_1=a_2c_2, \quad c_1a_1=c_2a_2\,, \label{In2}\\
&a_1c_2=a_2\ell_1,\quad c_2a_1=\ell_1a_2\,, \label{In3}\\
&a_1\ell_t=a_2\ell_{t+1}, \quad \ell_ta_1=\ell_{t+1}a_2\qquad\text{for}\quad t=1,\dots,n-3\,.\label{In4}
\end{align}
 {For $n\geq1$, we introduce the ideal $I_n$ of the path algebra $\Com\mathcal{Q}_n$ generated by}
\begin{enumerate}
 \item the  relation \eqref{In1} for $n=1$;
 \item the relation \eqref{In2} for $n=2$;
 \item the relations \eqref{In2} and \eqref{In3} for $n=3$;
 \item the relations \eqref{In2}, \eqref{In3} and \eqref{In4} for $n\geq 4$.
\end{enumerate}

We briefly recall the notion of \emph{Wemyss's reconstruction algebra} (see \cite[Section 2]{W2011}). First, for all integers $\alpha_i\geq2$, one    considers the labelled Dynkin diagram of type $A_l$:
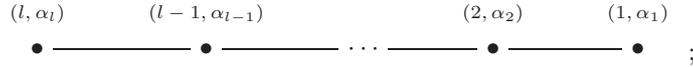
\[
 \xymatrix@R-2.3em{
\mbox{\scriptsize$(l,\alpha_l)$}&\mbox{\scriptsize$(l-1,\alpha_{l-1})$}&&\mbox{\scriptsize$(2,\alpha_2)$}&\mbox{\scriptsize$(1,\alpha_1)$}\\
\bullet\ar@{-}[r]&\bullet\ar@{-}[r]&\cdots\ar@{-}[r]&\bullet\ar@{-}[r]&\bullet
}
\xymatrix@R-2.5em{
\\ \\
;
} 
\]
second, one associates with this diagram the double quiver of the extended Dynkin quiver, by adding  a $0$-th vertex. One gets
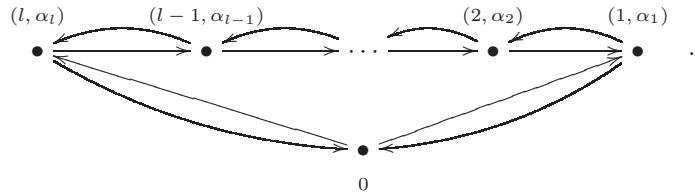
\[
 \xymatrix@R-2.3em{
\mbox{\scriptsize$(l,\alpha_l)$}&\mbox{\scriptsize$(l-1,\alpha_{l-1})$}&&\mbox{\scriptsize$(2,\alpha_2)$}&\mbox{\scriptsize$(1,\alpha_1)$}\\
\bullet\ar[r]\ar@/_2ex/[ddddrr]&\bullet\ar@/_2ex/[l]\ar[r]&\cdots\ar@/_2ex/[l]\ar[r]&\bullet\ar@/_2ex/[l]\ar[r]&\bullet\ar@/_2ex/[l]\ar@/^2ex/[ddddll]
\\ \\ \\ \\
&&\bullet\ar[uuuull]\ar[uuuurr]&&\\
&&\mbox{\scriptsize$0$}&&
}
\xymatrix@R-2.5em{
\\ \\
.
} 
\]

Finally, for any $\alpha_i>2$, one draws $\alpha_i-2$ additional arrows from the $i$-th vertex to the $0$-th vertex. We call $\mathcal{Q}$ the quiver we obtain at the end of this procedure. The \emph{recostruction algebra of type $A$} associated with the labels $[\alpha_1,\dots,\alpha_l]$, with each $\alpha_i\geq2$, was introduced in \cite[Definition 2.3]{W2011} as a certain quotient of the path algebra $\Com\mathcal{Q}$.

As in \cite[Definition 2.6]{W2011}, for $m_1,m_2\in\mathbb{N}$ with $\gcd(m_1,m_2)=1$ and $m_1>m_2$,  we denote by $\frac{1}{m_1}(1,m_2)$ the cyclic subgroup of $\GL(2,\Com)$ generated by
\[
 \begin{pmatrix}
   \varepsilon&0\\
   0&\varepsilon^{m_2}
 \end{pmatrix}\,,
\]
where $\varepsilon$ is a primitive $m_1$-th root of unity.
Furthermore, we recall that we can associate with a pair $m_1,m_2$ as above the so-called Jung-Hirzebruch continued fraction expansion of $\frac{m_1}{m_2}$, namely
\[
 \frac{m_1}{m_2}=\alpha_1-\frac{1}{\alpha_2-\displaystyle\frac{1}{\alpha_3-\frac{1}{(\dots)}}}\,,
\]
where each $\alpha_i\in\mathbb{N}_{\geq2}$.
\begin{defin}[\mbox{\cite[Definition 2.7]{W2011}}]
 The \emph{reconstruction algebra $A_{m_1,m_2}$} associated with the group $\frac{1}{m_1}(1,m_2)$ is the recostruction algebra of type $A$ corresponding to the integers of the Jung-Hirzebruch continued fraction expansion of $\frac{m_1}{m_2}$.
\end{defin}

By comparing the definitions, for $n\geq2$ the reconstruction algebras $A_{n,1}$ turn out to be precisely the quotients of the path algebras of our quivers $\mathcal{Q}_n$ by the ideals $I_n$. However, there is no such correspondence for $n=1$. In particular, we stress that, for $n\geq 2$, the quiver $\mathcal{Q}_n$ coincides with the special McKay quiver associated with toric singularity of type $\frac{1}{n}(1,1)$.

From now on, we definitively deviate from Wemyss's construction. We construct new quivers out of $\mathcal{Q}_n$, by adding a ``framing vertex:''
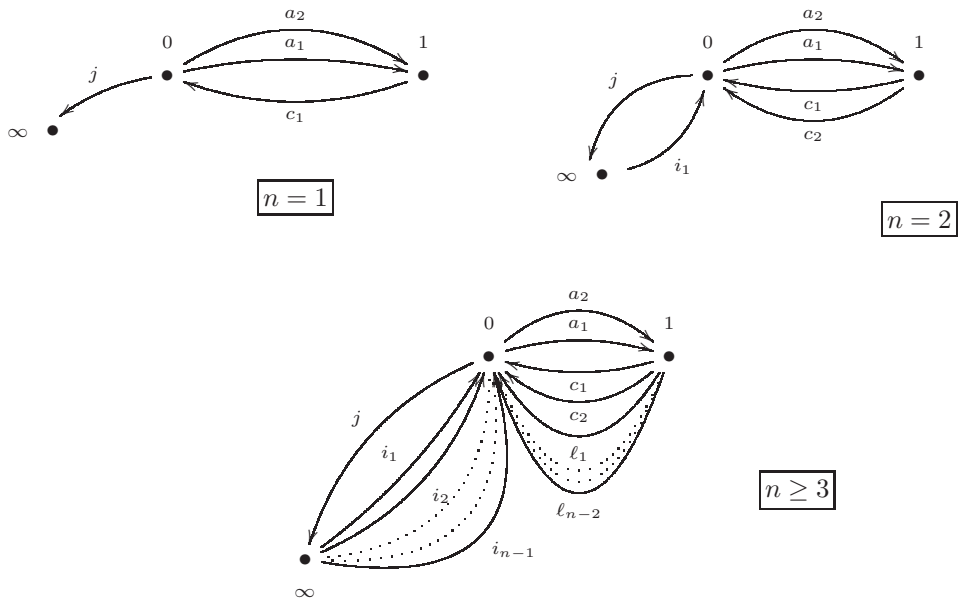
\begin{equation} \label{eqQnfr}
 \xymatrix@R-2.3em{
&\mbox{\scriptsize$0$}&&\mbox{\scriptsize$1$}\\
&\bullet\ar@/_1ex/[ldd]_{j}\ar@/^/[rr]^{a_{1}}\ar@/^4ex/[rr]^{a_{2}}&&\bullet\ar@/^10pt/[ll]^{c_{1}} \\ \\
\mbox{\scriptsize$\infty$}\ \ \bullet\ \ &&&\\
&&\\ &&\mbox{\framebox[1cm]{\begin{minipage}{1cm}\centering $n=1$\end{minipage}}}
}
\qquad\qquad
  \xymatrix@R-2.3em{
&\mbox{\scriptsize$0$}&&\mbox{\scriptsize$1$}\\
&\bullet\ar@/_3ex/[ldddd]_{j}\ar@/^/[rr]^{a_{1}}\ar@/^4ex/[rr]^{a_{2}}&&\bullet\ar@/^/[ll]^{c_{1}}
\ar@/^4ex/[ll]^{c_2} \\ \\ \\&&&&& \\ 
\mbox{\scriptsize$\infty$}\ \ \bullet\ \ar@/_12pt/[ruuuu]_{i_1}&&&&& \\
&&& \mbox{\framebox[1cm]{\begin{minipage}{1cm}\centering $n=2$\end{minipage}}}
}
\end{equation}
\begin{equation}
  \xymatrix@R-2.3em{
&&\mbox{\scriptsize$0$}&&\mbox{\scriptsize$1$}\\
&&\bullet\ar@/_3ex/[llddddddd]_{j}\ar@/^/[rr]^{a_{1}}\ar@/^4ex/[rr]^{a_{2}}&&\bullet\ar@/^/[ll]^{c_{1}}
\ar@/^4ex/[ll]^{c_2}\ar@/^7ex/[ll]^{\ell_1}\ar@{..}@/^10ex/[ll]\ar@{..}@/^11ex/[ll]
\ar@/^12ex/[ll]^{\ell_{n-2}}& \\ \\ \\ \\ \\&&&&&\mbox{\framebox[1cm]{\begin{minipage}{1cm}\centering $n\geq3$\end{minipage}}} \\ \\
\bullet\ar@/_/[rruuuuuuu]^{i_1}\ar@/_3ex/[rruuuuuuu]_{i_2}\ar@{..}@/_6ex/[rruuuuuuu]\ar@{..}@/_8ex/[rruuuuuuu]\ar@/_10ex/[rruuuuuuu]_{i_{n-1}}&&&&&\\
\mbox{\scriptsize$\infty$}&&&&
}
\end{equation}
We call these new quivers $\mathcal{Q}_n^{\operatorname{fr}}$, for $n\geq1$. {In what follows, any time an index runs from $1$ to $n-1$, for $n=1$ it will be understood that there is no associated object.}  {Consider the relations}
\begin{align}
&a_1c_1a_2=a_2c_1a_1\,, \label{In1f}\\
 &a_1c_1=a_2c_2, \quad c_1a_1+i_1j=c_2a_2\,, \label{In2f}\\
&a_1c_2=a_2\ell_1,\quad c_2a_1+i_2j=\ell_1a_2\,, \label{In3f}\\
&a_1\ell_t=a_2\ell_{t+1}, \quad \ell_ta_1+i_{t+2}j=\ell_{t+1}a_2\qquad\text{for}\quad t=1,\dots,n-3\,,\label{In4f}
\end{align}
 {and let $I^{\operatorname{fr}}_n$ be the ideal of the path algebra $\Com\mathcal{Q}_n^{\operatorname{fr}}$ generated by}
\begin{enumerate}
 \item the relation \eqref{In1f} for $n=1$;
 \item the relation \eqref{In2f} for $n=2$;
 \item the relations \eqref{In2f} and \eqref{In3f} for $n=3$;
 \item the relations \eqref{In2f}, \eqref{In3f} and \eqref{In4f} for $n\geq 4$.
\end{enumerate}

{Note that, while $I_1^{\operatorname{fr}}$ and $I_1$ are defined by the same relation, when $n\geq 2$ we not only added a framing vertex to $\mathcal{Q}_n$, but also ``framed'' the ideal $I_n$.}

Let $B_n^{\operatorname{fr}}=\Com\mathcal{Q}_n^{\operatorname{fr}}/I_n^{\operatorname{fr}}$; for every $\vec{v}:=(v_0,v_1)\in\mathbb{N}^2$, $w\in\mathbb{N}$, a $(\vec{v},w)$-dimensional representation of $B_n^{\operatorname{fr}}$ is given by the choice of three $\Com$-vector spaces, $V_0,V_1$ and $W$, with $\dim V_i = v_i$, $\dim W=w$, together with an element $(A_1,A_2;C_1,\dots,C_n;e;f_1,\dots,f_{n-1})$ of
\[
\Hom_\Com(V_0,V_1)^{\oplus2}\oplus\Hom_\Com(V_1,V_0)^{\oplus n}\oplus\Hom_\Com(V_0,W)\oplus\Hom_\Com(W,V_0)^{\oplus n-1}
\]
 compatible with the relations in \eqref{In1f}--\eqref{In4f}. We will refer to the totality of the equations induced by \eqref{In1f}--\eqref{In4f} at the representation level as ``condition (Q1).'' The vertex $\infty$ is interpreted as a framing  vertex  because we regard $\operatorname{Rep}(B_n^{\operatorname{fr}},\vec{v},w)$ as a $\GL(v_0,\Com)\times\GL(v_1,\Com)$-variety, avoiding the change of basis action of $\GL(w,\Com)$.
 
 \begin{remark}\label{quattrodue}
  Let $\mathcal{Q}$ be a quiver and denote by $\mathcal{Q}^{\operatorname{double}}$ its double, which is obtained from $\mathcal{Q}$ by adding for any arrow a new arrow in the opposite direction. Fix $\vec{v}=(v_i)\in\mathbb{N}^I$, and let $G_{\vec{v}}=\Pi_{i\in I}\GL(v_i)$. One can see \cite[\S 4.3]{Gin} that $\operatorname{Rep}(\mathcal{Q}^{\operatorname{double}},\vec{v})\simeq T^\ast\operatorname{Rep}(\mathcal{Q},\vec{v})$; in particular, this space carries a natural symplectic structure, the canonical one. This is the setting in which one would like to perform the so-called \emph{Hamiltonian reduction}: the first step of this procedure consists in considering the moment map $\mu$ associated with the natural (symplectic) $G_{\vec{v}}$-action on $\operatorname{Rep}(\mathcal{Q}^{\operatorname{double}},\vec{v})$, which is given by
  \begin{equation}\label{eq:moment}
   \xymatrix@R=2pt{
    \operatorname{Rep}(\mathcal{Q}^{\operatorname{double}},\vec{v})\ar[r]^-\mu&\mathfrak{g}_{\vec{v}}^\ast\simeq\Pi_{i\in I}\End(\Com^{v_i})\\
    X\ar@{|->}[r]&\sum_{a\in E}\left(X_a\circ X_{a^\ast}-X_{a^\ast}\circ X_a\right)\,,
   }
  \end{equation}
  where $a^\ast$ is the arrow opposite to $a$. Note that the moment map is induced, at the representation level, by a moment element, which by an abuse of notation we call $\mu$ again:
  \begin{equation}\label{eq:momel}
   \mu = \sum_{a \in E}(aa^*-a^*a) \in \Com\mathcal{Q}^{\operatorname{double}}\,.
  \end{equation}
One can see that the moment element can be decomposed as follows: $\mu = (\mu_i)_{i \in I}$, where $\mu_i\in e_i\cdot \Com\mathcal{Q}^{\operatorname{double}}\cdot e_i$. The zero level set $\mu^{-1}(0)$ can be reinterpreted as $\operatorname{Rep}(B,\vec{v})$, where $B$ is the quotient $\Com\mathcal{Q}^{\operatorname{double}}/\langle\mu\rangle$.
  
  For now we only observe that the algebra $B^{\operatorname{fr}}_2$ fits into this picture. We shall return to this topic in Section \ref{sec:Nak}. For $n\neq2$, although the relations defining the ideals $I^{\operatorname{fr}}_n$ look like equations for the zero level of some ``deformed (Poissonian?) moment map,'' at present we are not able to interpret them from this perspective.\label{rem:moment}
 \end{remark}

 \begin{defin}\label{def:stabil}
  Fix $\vartheta\in\mathbb R^2$. A $(\vec{v},w)$-dimensional representation  {$(V_0,V_1,W)$} is said to be \emph{$\vartheta$-semistable} if, for any sub-representation $S=(S_0,S_1)$ {$\subseteq (V_0,V_1)$}, one has:
\begin{gather}
 \text{if $S_0\subseteq \ker e$, then $\vartheta\cdot(\dim S_0,\dim S_1)\leq0$}\,;\label{eq:stabnak1}\\[5pt] 
 \text{if $S_0\supseteq \im f_i$ for $i=1,\dots,n-1$, then $\vartheta\cdot(\dim S_0,\dim S_1)\leq\vartheta\cdot(v_0,v_1)$}\,.\label{eq:stabnak2}
\end{gather}
{For $n=1$, the condition in \eqref{eq:stabnak2} must hold for any sub-representation.} A $\vartheta$-semistable representation is \emph{$\vartheta$-stable} if strict inequality holds in \eqref{eq:stabnak1} whenever $S\neq0$ and
in \eqref{eq:stabnak2} whenever $S\neq(V_0,V_1)$.
 \end{defin}
 
 {\begin{rem}
 The notion of (semi)stability introduced in Definition \ref{def:stabil} is indeed a (semi)stability in the sense of GIT (depending on the parameter $\vartheta$). One can see this by slightly generalizing a result due to Crawley-Boevey for usual framed quivers \cite[p.~261]{CrBo}, and then referring to King's classical paper \cite{king}. This procedure is completely straightforward and will be therefore omitted. However, the scrupulous reader can find all details in \cite[Section 3]{blr}.
\end{rem}
}
 
\subsection{The main result} We denote by $/\!/_{\!\vartheta}$ the GIT quotient associated with the parameter $\vartheta$.
We shall prove the following result:
\begin{thm}\label{thm:quiver1}
 For every $n,c\geq1$, the variety $\Hilb^c(X_n)$ is isomorphic to an irreducible connected component of the quotient
\begin{equation}\label{quotient}
 \operatorname{Rep}\left(B_n^{\operatorname{fr}},\vec{v}_c,1\right)^{ss}_{\vartheta_c}/\!/_{\!\vartheta_c}\,\GL(c,\Com)\times\GL(c,\Com)\,,
\end{equation}
where $\vec{v}_c=(c,c)$ and $\vartheta_c=(2c,-2c+1)$.

More precisely, this component is given by the equations $f_1=\dots=f_{n-1}=0$. {In particular, for $n=1$, it coincides with the whole space.}
\end{thm}
{
\begin{remark} In the case of the Hilbert scheme of $\mathbb C^2$, since the space $W$ is 1-dimensional, and due to the stability condition, Nakajima's map $j$ vanishes \cite{Nakabook,GanGin}. In the present case, that would correspond to the fact that the equations $f_i=0$ are implied by the stability conditions, and then the quotient \eqref{quotient} would be irreducible. At the moment we are unable to prove this result, and we only show that the equations $f_i=0$ follow from conditions (Q2) and (Q3*) (introduced below), which are a priori stronger than the stability conditions, which will be expressed by  (Q2) and (Q3).
\end{remark}
}

We fix $V_0=V_1=\Com^c$. The following Lemma is a direct consequence of the semistability conditions \eqref{eq:stabnak1} and \eqref{eq:stabnak2}.
\begin{lemma}\label{lemma:stabil}
 An element $(A_1,A_2;C_1,\dots,C_n;e;f_1,\dots,f_{n-1})\in \operatorname{Rep}(B_n^{\operatorname{fr}},\vec{v}_c,1)$ is $\vartheta_c$-semistable if and only if
\begin{itemize}
 \item[\rm{(Q2)}]
   for all sub-representations $S=(S_0,S_1)$ such that $S_0\subseteq\ker e$, one has $\dim S_0\leq\dim S_1$, and, if $\dim S_0=\dim S_1$, then $S=0$;
 \item[\rm{(Q3)}]
   for all sub-representations $S=(S_0,S_1)$ such that $S_0\supseteq\im f_i$, for $i=1,\dots,n-1$, one has $\dim S_0\leq\dim S_1$ {(for $n=1$, this must hold for any sub-representation)}.
\end{itemize}
Furthermore, $\vartheta_c$-semistability and $\vartheta_c$-stability are equivalent.
\end{lemma}

We denote by $\mathcal{R}_n(c)$ the space $\operatorname{Rep}\left(B_n^{\operatorname{fr}},\vec{v}_c,1\right)^{ss}_{\vartheta_c}$. 

\begin{cor}\label{cor:e0}
 If $(A_1,A_2;C_1,\dots,C_n;e;f_1,\dots,f_{n-1})\in\mathcal{R}_n(c)$, the map $e$ is not zero.
\end{cor}
\begin{proof}
 If $e$ is the zero map the sub-representation $S=(\Com^c,\Com^c)$  violates condition (Q2).
\end{proof}

We observe that the action of $\GL(c,\Com)\times\GL(c,\Com)$ on $\mathcal{R}_n(c)$ is compatible with the action of the same group on the space of ADHM data $P^n(c)$ we have defined in eq.~\eqref{eqrho}. Thus, to prove Theorem \ref{thm:quiver1}, we can work directly on $P^n(c)$ and $\mathcal{R}_n(c)$ without taking the actions into consideration.

We denote by $Z_n(c)$ the closed subvariety of $\mathcal{R}_n(c)$ cut by the equations $f_1=\dots=f_{n-1}=0$ {(for $n=1$, $Z_n(c)$ coincides with the whole $\mathcal{R}_n(c)$)}.

\begin{prop}\label{prop:PnZn}
 One has $P^n(c)=Z_n(c)$. In particular, this proves Theorem \ref{thm:quiver1} when $n=1$.
\end{prop}

{The proof of this Proposition will require the next two lemmas.}

\begin{lemma}\label{lemma:P2Q2*}
 The matrices $A_1,A_2$ satisfy  condition \mbox{\rm(P2)} if and only if they satisfy the requirement
\begin{itemize}
 \item[\rm{(Q3*)}]
 for any subspace $S_0\subseteq\Com^c$, $\dim (A_1(S_0)+A_2(S_0))\geq\dim S_0$.
\end{itemize}
\end{lemma}

\begin{proof}
Suppose that condition (P2) is satisfied by $A_1,A_2$. Let $S_0$ be any subspace, and let $\{v_1,\dots,v_k\}$ be a basis for it. Then, for suitable $[\nu_1,\nu_2]\in\Pu$, $\{(\nu_1A_1+\nu_2A_2)v_j\}_{j=1}^k$ is a set of linearly indipendent vectors in $A_1(S_0)+A_2(S_0)$. So (Q3*) is also satisfied.

To prove the converse, suppose that condition (P2) is not satisfied, i.e., the pencil $A_1+\lambda A_2$ is singular. Let us consider a polynomial solution of minimal degree $\varepsilon$ for that pencil,\footnote{By \emph{polynomial solution} we mean a solution $v(\lambda)$ of the equation $(A_1+\lambda A_2)v(\lambda)=0$ which is a polynomial in $\lambda$. Such a solution always exists \cite[p.~29]{Gan}.}
\begin{equation}\label{eq:vlam}
 v(\lambda)=v_0-\lambda v_1+\lambda^2v_2+\dots+(-1)^\varepsilon\lambda^\varepsilon v_\varepsilon,\qquad\text{with}\quad v_\varepsilon\neq0\,.
\end{equation}
Introduce the subspace $S_0:=\langle v_0,\dots,v_\varepsilon\rangle$. The vectors $v_0,\dots,v_\varepsilon$ are linearly indipendent 
(see \cite{Gan},  \S XII, Proof of Theorem 4), so  that $\dim S_0=\varepsilon+1$. Now,
\begin{equation}\label{eq:1}
 A_1(S_0)+A_2(S_0)=\langle A_1v_0,\dots,A_1v_\varepsilon,A_2v_0,\dots,A_2v_\varepsilon\rangle\,.
\end{equation}
By substituting \eqref{eq:vlam} into the equation $(A_1+\lambda A_2)v(\lambda)=0$ and by equating to zero the coefficients of the powers of $\lambda$, we get the $\varepsilon+2$ relations
\begin{equation}\label{eq:vlambda}
 A_1v_0=0\ ,\quad A_2v_0-A_1v_1=0\ ,\ \dots\ ,\quad A_2v_{\varepsilon-1}-A_1v_\varepsilon=0\ ,\quad A_2v_\varepsilon=0\,.
\end{equation}
Hence the maximum number of linearly indipendent vectors in \eqref{eq:1} is
\[
 2\varepsilon+2-(\varepsilon+2)=\varepsilon<\varepsilon+1\,.
\]
\end{proof}

\begin{lemma}
\label{lmKerA}
 If $\left(A_1,A_2;C_1,\dots,C_{n};e;f_1,\dots,f_{n-1}\right)\in\mathcal{R}_n(c)$, $\ker A_1\cap\ker A_2=\{0\}$.
\end{lemma}
\begin{proof}
 Suppose that there exists a nonzero vector $v\in\Com^{c}$ such that $A_{i}(v)=0$ for $i=1,2$. {For $n=1$, the sub-representation $(\langle v\rangle,\{0\})$ violates condition (Q3).} For $n\geq 2$, if $v\in\ker e$, the sub-representation $(\langle v\rangle,\{0\})$ violates condition (Q2); if $v\notin\ker e$, one has $\im f_q=\langle f_qe(v)\rangle$, for $q=1,\dots,n-1$, but $f_qe(v)=(f_qe+C_qA_1-C_{q+1}A_2)(v)=0$, so that the sub-representation $(\langle v\rangle,\{0\})$ violates condition (Q3).
\end{proof}

\begin{proof}[Proof of Proposition \ref{prop:PnZn}]
 First, observe that for a point in $Z_n(c)$ condition (P1) expresses exactly the constraints in \eqref{In1f}--\eqref{In4f}, and  (Q3) amounts to say that \emph{$\dim S_0\leq\dim S_1$ for all sub-representations $S=(S_0,S_1)$}. We begin by showing that $P^n(c)\subseteq Z_n(c)$, i.e., that any element of $P^n(c)$ satisfies conditions (Q2) and (Q3).
 
 Let $S=(S_0,S_1)$ be a sub-representation which makes condition (Q2) false. In particular,  as Lemma \ref{lemma:P2Q2*}  implies  $\dim S_0\le \dim S_1$,  one must have
 $\dim S_0=\dim S_1>0$, and $S_0\subseteq\ker e$. { By   Corollary \ref{useful}}, which works for the restrictions $A_1|_{S_0}$, $A_2|_{S_0}$, $C_1|_{S_1}$, \dots, $C_n|_{S_1}$ as well, we can produce a vector $0\neq v\in S_0$ and parameters $\lambda_1,\lambda_2,\mu_1,\mu_2$ that fail to satisfy condition (P3). As a consequence, condition (Q2) holds in $P^n(c)$.
 
Since condition (Q3*) is clearly stronger than (Q3), Lemma \ref{lemma:P2Q2*} entails that also condition (Q3) holds in $P^n(c)$. 

As for the opposite inclusion, $Z_n(c)\subseteq P^n(c)$, we have to show that any element of $Z_n(c)$ satisfies conditions (P2) and (P3).
Suppose that the pencil $A_{1}+\lambda A_{2}$ is singular. Let $v(\lambda)$ be a polynomial solution of minimal degree for the pencil defined in eq.\eqref{eq:vlam}. Lemma \ref{lmKerA} implies   $\varepsilon\geq1$. Set
\begin{equation*}
\begin{aligned}
S_{0}&:=\langle v_{0},\dots,v_{\varepsilon}\rangle\,,\\
S_{1}&:=A_{1}(V_{0})+A_{2}(V_{0})\,,\\
S_{2}&:=\sum_{q=1}^{n}C_{q}(V_{1})\,.
\end{aligned}
\end{equation*}
We know that $\dim S_0=\varepsilon+1$ and $\dim S_1=\varepsilon$ \cite[XII, Proof of Theorem 4]{Gan}. It is not difficult to show that $S_2=0$. In fact, by using repeatedly condition (P1) and the relations in \eqref{eq:vlambda}, one sees that
\[
 S_2=\langle C_nA_1v_1,\dots,C_nA_1v_\varepsilon\rangle\,,
\]
and, if $S_2\neq0$, by direct computation one can check that $\sum_{q=1}^\varepsilon (-\lambda)^{q-1}C_nA_1v_q$ is a polynomial solution of degree smaller than $\varepsilon$. Thus, the sub-representation $S=(S_{0},S_{1})$ fails to satisfy  (Q3). As a consequence, condition (P2) holds in $Z_n(c)$.

Finally, let $v\in\Com^c$ be a vector violating condition (P3). Set
\[
 S_0:=\langle v\rangle\,,\quad S_1:=\langle A_1v,A_2v\rangle\,.
\]
In particular, $v\neq0$, so $\dim S_0=1$. The conditions $\lambda_2A_1v+\lambda_1A_2v=0$ {and Lemma \ref{lmKerA}} together imply   $\dim S_1=1$. We claim that
\begin{equation}
S_2:=\sum_{q=1}^nC_q(S_2)\subseteq\langle v\rangle\,.
\label{eqVv}
\end{equation}
Suppose indeed that $S_1=\langle A_1v\rangle$. Then $A_{2}v=\lambda A_{1}v$ for some $\lambda\in\Com$. This implies  
\begin{equation*}
S_2=\langle C_{1}A_{1}v,\dots,C_{n}A_{1}v\rangle\,.
\end{equation*}
Now, by hypothesis $C_{n}A_{1}v\in\langle v\rangle$; one has
\[
 C_{q}A_{1}v=C_{q+1}A_{2}v=\lambda C_{q+1}A_{1}v\qquad\text{for $q=1,\dots,n-1$}\,,
\]
so that by induction one gets $C_{q}A_{1}v=\lambda^{n-q}C_nA_1v\subseteq\langle v\rangle$, for $q=1,\dots,n-1$. The case $S_1=\langle A_2v\rangle$ is completely analogous. Thus, the claim is proved, and as $S_0\subseteq\ker e$ by hypothesis,  $(S_0,S_1)$ is a sub-representation that violates condition (Q2). So $P^n(c)=Z_n(c)$. Note that  condition (P3) holds on the whole of $\mathcal{R}_n(c)$.
\end{proof}

\begin{proof}[Proof of Theorem \ref{thm:quiver1} when $n\geq2$]
 Having proved Proposition \ref{prop:PnZn}, it remains to show that $Z_n(c)$ is a connected component of $\mathcal{R}_n(c)$, i.e.,
 that $Z_n(c)$ is closed and open  in $\mathcal{R}_n(c)$, and   it is connected. This last statement follows  from the fact that $Z_n(c)=P^n(c)$ and from the connectedness of $\Hilb^c(X_n)$ \cite[Prop.~2.3]{fog68}.

We claim that the closed subvariety $Z_n(c)$ coincides with the open subset of $\mathcal{R}_n(c)$ where condition (P2) is satisfied, that is, we assume condition (P2) and prove that $f_1=\dots=f_{n-1}=0$.

Given an element $(A_1,A_2;C_1,\dots,C_n;e;f_1,\dots,f_{n-1})\in\mathcal{R}_n(c)$, we introduce the matrices $A_{1m}$, $A_{2m}$ and $E_m$  as in eq.~\eqref{eq410}; by  condition (P2), we can choose $m$ such that $\det A_{2m}\neq 0$. After introducing the matrix $B_m=A_{2m}^{-1}A_{1m}$, we define
\[
 u_m:=\sum_{q=1}^{n-1}\binom{n-2}{q-1}s_m^{n-1-q}c_m^{q-1}f_q\,,
\]
and   set
\[
 b_1={^tB_m}\ ,\qquad b_2={^tE_m}\ ,\qquad i={^te}\ ,\qquad j={^tu_m}\,.
\]
The data $b_1,b_2,i,j$ satisfy:
\begin{itemize}
 \item[(i)] $[b_1,b_2]+ij=0$;
 \item[(ii)]there exists no proper subspace $S\subsetneq\Com^c$ such that $b_\alpha(S)\subseteq S$ ($\alpha=1,2$) and $\im i\subseteq S$.
\end{itemize}
Indeed, relation (i) follows by direct computation, by suitably manipulating condition (Q1) and the expressions for the $C_q$'s given in eq.~\eqref{eq64}. As for condition (ii), it suffices to observe that the second statement of Lemma \ref{lmP3}, which applies here too, is equivalent to the maximality of the rank of
\[
 \begin{pmatrix}
  -(b_2-w\mathbf{1}_c)&(b_1-z\mathbf{1}_c)&i
 \end{pmatrix}\,,
\]
so that we can apply \cite[Lemma 2.7 (2)]{Nakabook}. By \cite[Prop.~2.8 (1)]{Nakabook}, one has $u_m=0$, which implies that $B_m$ and $E_m$ commute; by using this fact in combination with condition (Q1), one shows that $f_qe=0$, for $q=1,\dots,n-1$. Corollary \ref{cor:e0} allows one  to conclude that $f_1=\dots=f_{n-1}=0$, as wanted.
\end{proof}

\subsection{Comparison with Nakajima quiver varieties}\label{sec:Nak}
We focus now on the case $n=2$; in particular, as a consequence of Theorem \ref{thm:quiver1}, we recover a result for ALE spaces due to Kuznetsov \cite{kuz}.  We recall that, according to Nakajima \cite{naka1994}, any quiver $\mathcal{Q}$ with vertex set $I$ is associated with a quiver variety $\M_{\lambda,\vartheta}(\mathcal{Q},\vec{v},\vec{w})$, where $\vec{v}=(v_i),\vec{w}=(w_i)\in\mathbb{N}^{I}$, $\lambda=(\lambda_i)\in\Com^{I}$ and $\vartheta\in\mathbb R^{I}$. The main steps of this construction are the following (see \cite{Gin} for further details):
\begin{enumerate}
\item[(i)] one considers the quiver $\mathcal{Q}^{\operatorname{fr}}$, the framed version of $\mathcal{Q}$, which is obtained by taking as set of vertices the disjoint union $I\sqcup I'$, where $I'$ is a copy of $I$, and by adding to the set of arrows $E$   one arrow from the $i$-th vertex to the $i'$-th vertex, for any $i\in I$. We call $\{d_{i}\}_{i \in I}$ these new arrows;
 \item[(ii)] as in Remark \ref{quattrodue}, one introduces an auxiliary quiver $\mathcal{Q}^{\operatorname{fr},\operatorname{double}}$, the double of $\mathcal{Q}^{\operatorname{fr}}$, and consider the space $\operatorname{Rep}(\mathcal{Q}^{\operatorname{fr},\operatorname{double}},\vec{v},\vec{w})$ of $(\vec{v},\vec{w})$-dimensional representations of $\Com \mathcal{Q}^{\operatorname{fr},\operatorname{double}}$. The group $G_{\vec{v}} = \prod_{i\in I} \GL(v_i)$ acts naturally on this space if one regards $\GL(v_i)$ as the group of   automorphisms of the vector space associated with the $i$-th vertex (of the original quiver $\mathcal{Q}$);
\item[(iii)] since the action of $G_{\vec{v}}$ on $\operatorname{Rep}(\mathcal{Q}^{\operatorname{fr},\operatorname{double}},\vec{v},\vec{w})$ is symplectic, one can introduce a moment map 
\begin{equation}
\mu=\sum_{a\in E}\left(X_a\circ X_{a^\ast}-X_{a^\ast}\circ X_a\right)+\sum_{i \in I}X_{d_{i}^{*}}\circ X_{d_{i}}
\end{equation}
(cf. eq.~\eqref{eq:moment}), corresponding to a moment element, which we call again $\mu$ (analogous to eq.~\eqref{eq:momel});
 \item[(iv)] one defines the \emph{framed pre-projective algebra $\Pi^{\operatorname{fr}}_{\lambda}\mathcal{Q}$ of $\mathcal{Q}$ with parameter $\lambda$} as the quotient $\Com\mathcal{Q}^{\operatorname{fr},\operatorname{double}}/J$, where $J$ is the ideal of $\Com\mathcal{Q}^{\operatorname{fr},\operatorname{double}}$ generated by the elements $\{\mu_i-\lambda_i\}_{i\in I}$. The fibre
 $$\mu^{-1}(\sum_{i \in I}\lambda_i \mathbf{1}_{v_i})  \subset  \operatorname{Rep}(\mathcal{Q}^{\operatorname{fr},\operatorname{double}},\vec{v},\vec{w})$$
  is the space of $(\vec{v},\vec{w})$-dimensional representations of $\Pi^{\operatorname{fr}}_{\lambda}\mathcal{Q}$, which we denote  $\operatorname{Rep}(\Pi^{\operatorname{fr}}_{\lambda}\mathcal{Q},\vec{v},\vec{w})$;
 \item[(v)]
one defines  the quotient $\M_{\lambda,\vartheta}(\mathcal{Q},\vec{v},\vec{w}):=\operatorname{Rep}(\Pi_{\lambda}^{\operatorname{fr}}\mathcal{{Q}},\vec{v},\vec{w})^{\operatorname{ss}}_\vartheta/\!/_{\!\vartheta}\, G_{\vec{v}}$.
\end{enumerate}
We denote by $\mathcal{A}$ the affine Dynkin quiver of type $A_1^{(1)}$, namely,

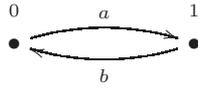
\[
 \xymatrix@R-2.3em{
   \mbox{\scriptsize$0$}&&\mbox{\scriptsize$1$}\\
   \bullet\ar@/^/[rr]^a&&\bullet\ar@/^/[ll]^b
 }
\xymatrix@R-2.3em{
  \\
  \\
  \\
    }
\]

\begin{cor}\label{cor:quiver}
For every $c\geq1$, the variety $\Hilb^c(X_2)$ is isomorphic to the Nakajima quiver variety $\mathcal{M}_{0,\vartheta_c}(\mathcal{A},\vec{v}_c,\vec{w})$, where $\vartheta_c=(2c,-2c+1)$, $\vec{v}_c=(c,c)$ and $\vec{w}=(1,0)$.
\end{cor}
\begin{proof}
As recalled above, one wants to consider the quiver $\mathcal{A}^{\operatorname{fr},\operatorname{double}}$:
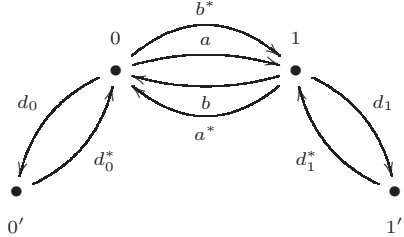
\begin{equation}
 \xymatrix@R-2.3em{
&\mbox{\scriptsize$0$}&&\mbox{\scriptsize$1$}\\
&\bullet\ar@/_2ex/[lddddd]_{d_{0}}\ar@/^/[rr]^{a}\ar@/^4ex/[rr]^{b^{*}}&&\bullet\ar@/^/[ll]^{b}\ar@/^4ex/[ll]^{a^{*}} \ar@/^2ex/[rddddd]^{d_{1}}  \\ \\ \\ \\ \\ 
\bullet\ar@/_2ex/[ruuuuu]_{d_0^{*}}&&&& \bullet \ar@/^2ex/[luuuuu]^{d_1^{*}}\\
\mbox{\scriptsize$0'$}&&&& \mbox{\scriptsize$1'$}
}
\label{eqK}
\end{equation}
The choice $\vec{w}=(1,0)$ implies that the linear morphisms associated with $d_{1}$ and $d^{*}_{1}$ vanish, and this allows one to construct $\M_{\lambda,\vartheta}(\mathcal{A},\vec{v},\vec{w})$ using $\mathcal{Q}_{2}^{\textmd{fr}}$ (see \eqref{eqQnfr}). By a general result proved by Crawley-Boevey \cite{CrBo}, the variety $\mathcal{M}_{0,\vartheta_c}(\mathcal{A},\vec{v}_c,\vec{w})$ is connected (see also \cite[Theorem~5.2.2]{Gin} for some comments). The thesis now follows from Theorem \ref{thm:quiver1}.
\end{proof}

\bigskip

\appendix \section{Proof of Proposition \ref{pro1}}
\label{appa} 

\subsection{A first step}  {We recall from Section \ref{background} the notation $G_{\vec{k}}=\Aut(\Uk)\times\Aut(\Vk)\times\Aut(\Wk)$.
Moreover, there we introduced a principal $G_{\vec{k}}$-bundle $P_{\vec{k}}$ such that 
$\Hilb^c(X_n)$ can be described   as a quotient $P_{\vec{k}}/G_{\vec{k}}$ for $\vec{k}=(n,1,0,c)$ (from now on we fix this value of $\vec{k}$). We recall also that} the group $\GL(c, \Com)$ acts on the space  $\T(c)$ according to the rule
\eqref{eqGL(c)onK}. Moreover,  by embedding $\GL(c, \Com)$  into  $\Gk$  by  the equations
\begin{equation}\label{eqiota}
\iota\colon \phi \mapsto 
\left(^t\phi^{-1},\,
\operatorname{diag} ({^t\phi^{-1}},\, ^t\phi^{-1}, \, 1),\,
^t\phi^{-1}\right)\,,
\end{equation}
the group   $\GL(c, \Com)$  also acts on $\Pk $. We can therefore state the following Lemma.

\begin{lemma}\label{lmclimm}
There is a $\GL(c, \Com)$-equivariant closed immersion 
$j_{m}\colon \T(c) \to P_{\vec{k},m}$,
where $\{P_{\vec{k},m}\} $ is the open cover of  $\Pk $ given by the inverse image of the open cover  {$\{U^{n,c}_m\}$ of $\operatorname{Hilb}^c(X_n)$ introduced in eq.~\eqref{U}}.
\end{lemma}
This immersion  induces an isomorphism
\begin{equation}\eta_m\colon\T(c)/\GL(c, \Com)\longrightarrow P_{\vec{k},m}/\Gk\simeq U^{n,c}_m\,.
\label{eqeta}
\end{equation}
We start by introducing some constructions that will be used in the proof of this Lemma. In particular,  we  define the immersion $j_{m}$ for any given $m\in\{0,\dots,c\}$. To this aim, after fixing homogeneous coordinates $[y_1,y_2]$ for $\Pu$ (cfr. eq.~\eqref{eqSn}), we introduce additional $c$ pairs of coordinates 
\begin{equation*}
[y_{1m},y_{2m}]=[c_{m}y_{1}+s_{m}y_{2},-s_{m}y_{1}+c_{m}y_{2}]\qquad m=0,\dots,c\,,
\label{beta10}
\end{equation*}
where  $c_m$ and $s_m$ are the real numbers defined in eq.~\eqref{eqcmsm}. The set $\left\{y_{2m}^{q}y_{1m}^{h-q}\right\}_{q=0}^{h}$ is a basis for $H^{0}\left(\On(0,h)\right)=H^{0}\left(\pi^{*}\Ol_{\Pu}(h)\right)$ for all $h\geq1$, where $\pi\colon\Sigma_{n}\longrightarrow\Pu$ is the canonical projection. Furthermore  the (unique up to homotheties) global section  $s_{E}$  of $\On(E)$ induces an injection $\On(0,n)\rightarrowtail\On(1,0)$, so that the set $\left\{(y_{2m}^{q}y_{1m}^{n-q})s_{E}\right\}_{q=0}^{n}\cup\{s_{\infty}\}$ is a basis for $H^{0}\left(\On(1,0)\right)$, where $s_{\infty}$ is the section characterized by the condition $\{s_{\infty}=0\}=\li$. This notation allows us to expand the morphisms $\alpha$ and $\beta$ (see eq. \eqref{fundamentalmonad}) as follows:
\begin{equation}
\begin{aligned}
\alpha&=
\begin{pmatrix}
\sum_{q=0}^{n}\alpha_{1q}^{(m)}\left(y_{2m}^{q}y_{1m}^{n-q}s_{E}\right)+\alpha_{1,n+1}s_{\infty}\\[7pt]
\alpha_{20}^{(m)}y_{1m}+\alpha_{21}^{(m)}y_{2m}
\end{pmatrix}\\
\beta&=
\begin{pmatrix}
\beta_{10}^{(m)}y_{1m}+\beta_{11}^{(m)}y_{2m} &
\sum_{q=0}^{n}\beta_{2q}^{(m)}\left(y_{2m}^{q}y_{1m}^{n-q}s_{E}\right)+\beta_{2,n+1}s_{\infty}
\end{pmatrix}\,.
\end{aligned}
\label{eqab}
\end{equation}
Note that the choice of a framing for a torsion-free sheaf $\E$ which is trivial at infinity, and is the cohomology of a monad of the type \eqref{fundamentalmonad}, is equivalent to the choice of a basis for $H^{0}(\E|_{\ell_{\infty}})\simeq H^0(\ker\beta|_{\ell_\infty})=\ker H^0(\beta|_{\ell_\infty})$, i.e.,  it is given by an injective linear map
\begin{equation}
 \xi\colon\Com^r\to H^{0}(\Vk|_{\li}) \quad\text{such that}\quad H^0(\beta|_{\ell_\infty})\circ\xi=0\,.
\end{equation}

We put $V_{\vec{k}}:=H^{0}(\Vk|_{\li})$.  We can   define the morphism
\begin{equation*}
\begin{aligned}
\tilde{\jmath}_{m}\colon\End(\Com^{c})^{\oplus 2}\oplus\Hom(\Com^{c},\Com)&\lra \Hom(\Uk,\Vk)\oplus\Hom(\Vk,\Wk)\oplus\Hom(\Com^{r},V_{\vec{k}})\\
(b_{1},b_{2},e)&\longmapsto (\alpha,\beta,\xi)\,,
\end{aligned}
\end{equation*}
where
$$\alpha=
\begin{pmatrix}
\bm{1}_c(y_{2m}^{n}s_{E})+{^tb_{2}}s_\infty\\
\bm{1}_cy_{1m}+{^tb_{1}}y_{2m}\\
0
\end{pmatrix}\,,$$
$$\beta=
\begin{pmatrix}
\bm{1}_cy_{1m}+{^tb_1}y_{2m},&
-\left(\bm{1}_c(y_{2m}^{n}s_{E})+{^tb_2}s_\infty\right),&
{^t}e\,s_\infty
\end{pmatrix}\,,
$$
$$\xi=
\begin{pmatrix}
0\\ \vdots \\ 0 \\
1
\end{pmatrix}\,,
$$
and $j_{m}$ is the restriction of $\tilde{\jmath}_{m}$ to $\T(c)$.

 \noindent {\em Proof of Lemma \ref{lmclimm}.} 
It is quite  clear that $\tilde{\jmath}_{m}$ is a closed immersion,  so that it is enough to prove that
\begin{equation*}
\im\tilde{\jmath}_{m}\cap P_{\vec{k},m}=\im j_{m}\,.
\end{equation*}

Let $(\alpha,\beta,\xi)=\tilde{\jmath}_{m}(b_{1},b_{2},e)$ be a point in the intersection $\im\tilde{\jmath}_{m}\cap P_{\vec{k},m}$;
the equation $\beta\circ\alpha=0$ implies that the triple $(b_{1},b_{2},e)$ satisfies condition (T1), while the fact that $\beta\otimes k(x)$ has maximal rank for all $x\in\Sigma_{n}$ entails condition (T2). It follows that
\begin{equation*}
\im\tilde{\jmath}_{m}\cap P_{\vec{k},m}\subseteq\im j_{m}\,.
\end{equation*}
To get the opposite inclusion, note that for all $(\alpha,\beta,\xi)\in\im\tilde{\jmath}_{m}$:
\begin{itemize}
\item[({i})] the morphism $\alpha\otimes k(x)$ fails to have maximal rank at most at a finite number of points $x\in\Sigma_{n}$; hence, $\alpha$ is injective as a sheaf morphism;
\item[({ii})] the morphisms $\alpha\otimes k(x)$ and $\beta\otimes k(x)$ have maximal rank for all points $x\in\li\cup F_{m}$;
\item[({iii})] the natural morphism $\Phi\colon H^0((\coker \alpha)_{\vert \ell_\infty}(-1))
\to H^0(\mathcal W_{\vec{k}}\vert_{\ell_\infty}(-1))$ is invertible;
\item[({iv})] $\beta_{1}|_{F_{m}}=\bm{1}_{c}$, where $\beta_{1}\colon\On(1,-1)^{\oplus k_2}\lra\On(1,0)^{\oplus k_{3}}$ is the first component of $\beta$;
\item[({v})] the morphism $\xi$ has maximal rank.
\end{itemize}
If $(\alpha,\beta,\xi)\in\im j_{m}$, condition (T2) implies that $\beta\otimes k(x)$ has maximal rank for all $x\in\Sigma_{n}\setminus(\li\cup F_{m})$; by  ({ii}) this is enough to ensure that $\beta$ is surjective. Condition (T1) implies   $\beta\circ\alpha=0$, so that we can define $\E=\ker\beta/\im\alpha$. By   ({i}) $\E$ is torsion free, by   ({ii}) and ({iii}) it is trivial at infinity, and by  ({iv}) $\E|_{F_{m}}$ is trivial as well.
The $\GL(c, \Com)$-equivariance of $j_{m}$ is readily checked.\qed

\subsection{A technical Lemma}  The next Lemma and Corollary will be used to  show that  $j_{m}$ induces an isomorphism between the quotients of $\T(c)$ and $P_{\vec{k},m}$ under the actions of $\GL(c, \Com)$ and $\Gk$, respectively (cf.~eq.~\eqref{eqeta}). Let  $X$ be a smooth algebraic variety over $\Com$ with an action
 {$\gamma\colon X\times G \to  X$}
 of a complex affine algebraic group $G$. The set-theoretical quotient $X/G$ has a natural structure of ringed space induced by the quotient map $q\colon X\lra X/G$.
 If the action is free, and    { the graph  morphism $X\times G\to X\times X$ 
is a closed immersion, }  $X/G$ is a smooth algebraic variety, the pair $(X/G,q)$ is a geometric quotient of $X$ modulo $G$, and
$X$ is a (locally isotrivial) principal $G$-bundle over $X/G$.
This can be proved by arguing as in the proof of  \cite[Theorem.~5.1]{bbr}.

Let $Y$ be a smooth closed subvariety of $X$ and let $H\stackrel{\iota}{\hookrightarrow}G$ be a closed subgroup of $G$. Assume that $H$ acts on $Y$ and that the inclusion $j\colon Y \hookrightarrow X$ is $H$-equivariant.   We consider the quotient
 $p\colon Y\lra Y/H$ as a ringed space with the quotient topology, and structure sheaf given by the sheaf of invariant functions.
 \begin{lemma}
 \label{lemmaquot}
 If
the intersection of $\im j$ with every $G$-orbit in $X$ is nonempty, and 
for all $G$-orbits $O_{G}$ in $X$, one has $\Stab_{G}(O_{G}\cap\im j)=\im\iota$,
then $j$ induces an isomorphism $\bar{\jmath}\colon Y/H \longrightarrow X/G$ of algebraic varieties.
\end{lemma}

\begin{proof}

By \cite[Prop.~0.7]{Mum}   the morphism $q$ is affine. If  $U\subset X/G$ is an open affine subset, then $V=q^{-1}(U)$ is   affine, $V=\Spec A$, so that $U= \Spec(A^{G})$, and the restricted morphism $q|_{V}$ is induced by the canonical injection
 $q^{\sharp}\colon A ^{G}\hookrightarrow A$. 
Since $j$ is an affine morphism \cite[Prop.~1.6.2.(i)]{EGA2}, the counterimage $W=j^{-1}(V)$ is affine, $W=\Spec B$, and by the equivariance of $j$ it is $H$-invariant. It follows that its image $p(W)=\Spec(B^{H})$ is affine, and the restricted morphism $p|_{W}$ is induced by the canonical injection $p^{\sharp}\colon B^{H}\hookrightarrow B$. 
Let $j^{\sharp}\colon A \to B $ be the homomorphism associated with $j$.  One can prove that $\im\left(j^{\sharp}\circ q^{\sharp}\right)\subseteq A^{G}$, and that this composition is an isomorphism, which induces $\bar{\jmath}_{\vert V}$. Thus $\bar\jmath$ is an isomorphism locally on the target, hence it is   an isomorphism.
\end{proof}

\begin{cor}
The morphism $p\colon Y\to Y/H$ is an $H$-principal bundle, and is a reduction of the structure group of $q\colon X\lra X/G$.
If  $X \to X/G$  is locally trivial, the same is true for $Y\to Y/H$.
\end{cor}

\subsection{Conclusion}

\begin{lemma}\label{lemmaquot2}
For any $\Gk$-orbit $O$ in $P_{\vec{k},m}$, the intersection $O\cap\im j_{m}$ is not empty and  its stabilizer  in $\Gk$ coincides with $\im\iota$
(the   morphism $\iota$  was defined in eq.\ \eqref{eqiota}).
\end{lemma}
\begin{proof}
Let $(\alpha,\beta,\xi)\in\Pk$ be any point and let $O$ be its $\Gk$-orbit. We call $\E$ the cohomology of $M(\alpha,\beta)$. One can verify that the condition $\E_{|F_{m}}\simeq \Ol_{F_{m}}$ is equivalent to the condition $\det(\beta_{10}^{(m)})\neq 0$ (see eq. \eqref{eqab}). By acting with $\Gk$, 
one can   find a point in $O$ such that
\begin{equation}
\left\{
\begin{aligned}
\beta_{10}^{(m)}&=\bm{1}_{c}\\
\beta_{2q}^{(m)}&=0\qquad q=0,\dots,n-1\,.
\end{aligned}
\right.
\label{eqImj1}
\end{equation}
We call $O_{1}$ the subvariety cut by these equations  inside $O$. The stabilizer of $O_{1}$ inside $\Gk$ is the closed subgroup $G_{1}$ characterized by the conditions $\psi_{11}=\chi$ and $\psi_{12}=0$.
 
We put $b_1:={}^{t}\beta_{11}^{(m)}$.

For all points in $O_{1}$ the equation $\beta\circ\alpha=0$ implies  
\begin{equation}\label{eq3}
\begin{cases}
 \alpha_{1q}^{(m)}=0 & q=0,\dots,n-1\\
 \alpha_{1n}^{(m)}=-\beta_{2n}^{(m)}\alpha_{20}^{(m)}\,.
\end{cases}
\end{equation}
In particular, for all points in $O_{1}$, the invertibility of $\Phi$ is equivalent to the condition $\det(\alpha_{1n}^{(m)})\neq0$, and by acting with $G_{1}$ we can find a point in $O_{1}$ such that
\begin{equation}
\alpha_{1n}^{(m)}=\bm{1}_c\,.
\label{eqImj2}
\end{equation}
We call $O_{2}$ the subvariety cut by this equation  inside $O_{1}$ . The stabilizer of $O_{2}$ in $G_{1}$ is the closed subgroup $G_{2}$ characterized by the condition $\chi=\phi$.

From eq.~\eqref{eq3} we deduce that $\rk\beta_{2n}^{(m)}=\rk\alpha_{20}^{(m)}=c$. By acting with $G_{2}$ we can find a point in $O_{2}$ such that
\begin{equation}
\alpha_{20}^{(m)}=
\begin{pmatrix}
\bm{1}_c\\
0
\end{pmatrix}\qquad,\qquad\beta_{2n}^{(m)}=
\begin{pmatrix}
-\bm{1}_c & 0
\end{pmatrix}
\label{eqImj3}
\end{equation}
We call $O_{3}$ the subvariety cut  by these equations inside $O_{2}$. The stabilizer of $O_{3}$ in $G_{2}$ is the closed subgroup $G_{3}$ characterized by the condition $\psi_{22}=
\left(\begin{smallmatrix}
\phi & 0\\
0 & \lambda
\end{smallmatrix}\right)$ for some $\lambda\in\Com^{*}$.

For all points in $O_{3}$ the equation $H^{0}\left(\beta|_{\li}\right)\circ\xi=0$ implies   ${}^{t}\xi=(0,\dots,0,\omega)$ for some $\omega\in\Com^{*}$.
By acting with $G_{3}$ we can find a point in $O_{3}$ such that
\begin{equation}
{}^{t}\xi=(0,\dots,0,1)
\label{eqImj4}
\end{equation}
We call $O_{4}$ the subvariety cut by this equation  inside $O_{3}$. The stabilizer of $O_{4}$ in $G_{3}$ is the closed subgroup $G_{4}$ characterized by the condition $\psi_{22}=
\left(\begin{smallmatrix}
\phi & 0\\
0 & 1
\end{smallmatrix}\right)$.  One can see that $G_{4}=\im\iota$. To get the thesis we have to prove that the subvariety $Z$ cut in $P_{\vec{k},m}$ by eqs. \eqref{eqImj1}, \eqref{eqImj2}, \eqref{eqImj3} and \eqref{eqImj4} coincides with $\im j_{m}$. For all points in $Z$ the equation $\beta\circ\alpha=0$ implies  
\begin{gather}
\alpha_{21}^{(m)}=
\begin{pmatrix}
{}^{t}b_1\\
{}^{t}e_2
\end{pmatrix}
\qquad,\qquad\beta_{2,n+1}=
\begin{pmatrix}
-\alpha_{1,n+1} & {}^{t}e
\end{pmatrix}\\
\text{and}\qquad [{}^{t}\alpha_{1,n+1},b_1]+e_2e=0\,,
\end{gather}
for some $e\in\Hom(\Com^{c},\Com)$ and $e_{2}\in\Hom(\Com,\Com^{c})$. We put $b_{2}={}^{t}\alpha_{1,n+1}$.  {By carefully adapting the arguments of the proof of \cite[Prop.~2.8.(1)]{Nakabook} for co-stable ADHM data,} one gets $e_{2}=0$. The equality $Z=\im j_{m}$ follows.
\end{proof}

So, as we anticipated, we have:

\begin{prop}
The morphism $j_m$ induces an isomorphism $ \eta_m\colon\T(c)/\GL(c, \Com)\longrightarrow P_{\vec{k},m}/\Gk\simeq U^{n,c}_m$.
\end{prop}

\begin{proof} This is proved by Lemma \ref{lemmaquot}, whose hypotheses are satisfied in view of Lemma \ref{lemmaquot2}.
\end{proof}

As a further step in the proof of Proposition \ref{pro1}, we introduce the open subsets
\begin{equation}\label{tcml}
\T(c)_{m,l}=j_{m}^{-1}\left(\im j_{m}\cap P_{\vec{k},l}\right)\qquad\text{for}\quad m,l=0,\dots,c\,.
\end{equation}
It is not difficult to see that 
\begin{equation}\label{eqtcml}
\T(c)_{m,l}=\left\{(b_{1},b_{2},e)\in \T(c)\left| \det\left(c_{m-l}\bm{1}_c-s_{m-l}b_{1}\right)\neq0\right.\right\}\,.
\end{equation}

To conclude our reasoning we need one more Lemma.
\begin{lemma}
For any $l,m=0,\dots,c$ and for any point $\vec{b}_{m}=(b_{1m},b_{2m},e_{m})\in \T(c)_{m,l}$, there exists a unique element $\psi_{l}(\vec{b}_{m})=(\phi,\psi,\chi)\in\Gk$ such that
 $\chi=\bm{1}_{c}$,
and 
the point $(\alpha',\beta',\xi')=\psi_{l}(\vec{b}_{m})\cdot j_{m}(\vec{b}_{m})$ lies in the image of $j_{l}$.
If we set \[(b_{1l},b_{2l},e_{l})=j_{l}^{-1}(\alpha',\beta',\xi')\,,\] we have
\begin{equation}
\left\{
\begin{aligned}
b_{1l}&=\left(c_{m-l}\bm{1}_c-s_{m-l}b_{1m}\right)^{-1}\left(s_{m-l}\bm{1}_c+c_{m-l}b_{1m}\right)\\
b_{2l}&=\left(c_{m-l}\bm{1}_c-s_{m-l}b_{1m}\right)^{n}b_{2m}\\
e_{l}&=e_{m}\,.
\end{aligned}
\right.
\label{eqbbel}
\end{equation}
\end{lemma}
\begin{proof}
If we set $(\alpha,\beta,\xi)=j_{m}(\vec{b}_{m})$, by expressing the coordinates $[y_{1m},y_{2m}]$ as functions of $[y_{1l},y_{2l}]$ we get
\begin{equation*}
\begin{aligned}
\alpha&=
\begin{pmatrix}
\sum_{q=0}^{n}(\sigma_{q}\bm{1}_{c})(y_{2l}^{q}y_{1l}^{n-q}s_{E})+{^tb_{2m}}s_\infty\\
d_{1m}y_{1m}+d_{2m}y_{2m}\\
0
\end{pmatrix}\,,\\
\beta&=
\begin{pmatrix}
d_{1m}y_{1m}+d_{2m}y_{2m},&
-\sum_{q=0}^{n}(\sigma_{q}\bm{1}_{c})(y_{2l}^{q}y_{1l}^{n-q}s_{E})-{^tb_{2m}}s_\infty,&
^te_{m}s_\infty
\end{pmatrix}\,,
\end{aligned}
\end{equation*}
where
\begin{equation*}
d_{1m}=c_{m-l}\bm{1}_c-s_{m-l}{}^t\mspace{2mu}b_{1m}\qquad\qquad d_{2m}=s_{m-l}\bm{1}+c_{m-l}{}^t\mspace{2mu}b_{1m}
\end{equation*}
and we have put $\sigma_{q}=\sigma^{n}_{l-m;nq}$ for $q=0,\dots,n$ (see eq. \eqref{eqsigma}). The explicit form of $\psi_{l}(\vec{b}_{m})$ is obtained by imposing the equality 
\begin{equation}
(\phi,\psi,\bm{1}_{c})\cdot(\alpha,\beta,\xi)=j_{l}(b_{1l},b_{2l},e_{l})
\label{eqbpsi}
\end{equation}
for some $(b_{1l},b_{2l},e_{l})\in \T(c)_{l}$. One gets
\begin{gather*}
\begin{aligned}
\phi&=d_{1m}^{-(n-1)}\\
\psi&=
\begin{pmatrix}
d_{1m} & \psi_{12,1} & 0\\
0 & d_{1m}^{-n} & 0\\
0 & 0 & 1
\end{pmatrix}\,,
\end{aligned}
\\[5pt]
\text{where}\qquad\psi_{12,1}=-\sum_{q=0}^{n-1}\sum_{p=0}^{q}\sigma_{q-p}\left(-d_{2m}d_{1m}^{-1}\right)^{p}y_{1l}^{q}y_{2l}^{n-1-q}\,.
\end{gather*}
Eq.~\eqref{eqbbel} follows  from eq.~\eqref{eqbpsi}.
\end{proof}
 
Equations \eqref{eqtcml} and \eqref{eqbbel}
yield a proof of Proposition \ref{pro1}.

\frenchspacing\bigskip

 \def\cprime{$'$} \def\cprime{$'$} \def\cprime{$'$} \def\cprime{$'$}


\begin{thebibliography}{MMM}

 
 

\bibitem{AOSV}
{\sc M. Aganagic, H. Ooguri, N. Saulina, and C. Vafa}, {\em Black holes,
  {$q$}-deformed 2d {Y}ang-{M}ills, and non-perturbative topological strings},
  Nuclear Phys. B {\bf 715} (2005),  304--348.

\bibitem{bbr}
{\sc C. Bartocci, U. Bruzzo, and C. L. S. Rava}, {\em Monads for framed sheaves
  on {H}irzebruch sur\-faces,} Adv. Geom. {\bf 15} (2015),   55--76 (revised version, with a mistake
  in a preliminary technical result corrected, in {\tt  arXiv:1504.02987v5} [math.AG]).

\bibitem{blr}
{\sc C. Bartocci, V. Lanza, and C. L. S. Rava}, {\em Moduli spaces of framed sheaves and quiver varieties,} {\tt arXiv:1610.0273 [math.AG]}.  
  
 \bibitem{Bie}  {\sc R. Bielawski},  {\em Quivers and Poisson structures}, Manuscripta Math. {\bf 141} (2013), 29--49. 
  
  \bibitem{Bot}
{\sc F. Bottacin}, {\em Poisson structures on {H}ilbert schemes of points
of a surface and integrable systems,} Manuscripta Math. {\bf 97} (1998), 517--527.

  \bibitem{BPSS}   {\sc U. Bruzzo, M. Pedrini, F. Sala and  R. J. Szabo}, {\em Framed sheaves on root stacks and supersymmetric gauge theories on ALE spaces,}     Adv. Math. {\bf 288} (2016), 1175--1308.

\bibitem{BPT}
{\sc U. Bruzzo, R. Poghossian, and A. Tanzini}, {\em Poincar\'e polynomial of
  moduli spaces of framed sheaves on (stacky) {H}irzebruch surfaces}, Comm.
  Math. Phys. {\bf 304} (2011),  395--409.
  

  
 \bibitem{BS} {\sc U. Bruzzo and F. Sala,} {\em Framed sheaves on projective stacks}, Adv. Math. {\bf  272} (2015),  20--95.
 
 \bibitem{BSS} {\sc U. Bruzzo, F. Sala and  R. J. Szabo},  {\em $\mathcal N\!\!=\!\!2$ quiver gauge theories on A-type ALE spaces,}  Lett. Math. Phys.    {\bf 105}  (2015),  401--445.

\bibitem{Bu}
{\sc N. P. Buchdahl}, {\em Stable {$2$}-bundles on {H}irzebruch surfaces},
  Math. Z. {\bf 194} (1987), 143--152.

\bibitem{CrBo}
{\sc W. Crawley-Boevey}, {\em Geometry of the moment map for representations of quivers}, Compositio Math. {\bf 126} (2001), 257--293.

\bibitem{fog68}
{\sc J. Fogarty}, {\em Algebraic families on an algebraic surface}, Am. J. Math. {\bf 10} (1968), 511--521.

\bibitem{FMP}
{\sc F. Fucito, J. F. Morales, and R. Poghossian}, {\em {Multi instanton
  calculus on ALE spaces}}, Nucl. Phys B {\bf  703} (2004),   518--536.
  
 \bibitem{GanGin} {\sc W. L. Gan and V. Ginzburg,} 
 {\em Almost-commuting variety, $\mathcal D$-modules,
and Cherednik algebras,} Int. Math. Res. Papers
Vol. 2006, Article ID 26439, pp. 1--54.

\bibitem{Gan} {\sc F. R. Gantmacher}, \newblock {\em The theory of matrices}, vol. 2, Chelsea Publishing Company, New York 1959.

\bibitem{Gin}
{\sc V. Ginzburg}, {\em Lectures on Nakajima's Quiver Varieties},
Geometric methods in representation theory. I, 145Ð219, S\'emin.
Congr., 24-I, Soc. Math. France, Paris, 2012. 

\bibitem{GSST}
{\sc L. Griguolo, D. Seminara, R. J. Szabo, and A. Tanzini}, {\em Black holes,
  instanton counting on toric singularities and {$q$}-deformed two-dimensional
  {Y}ang-{M}ills theory}, Nuclear Phys. B {\bf  772} (2007),   1--24.

\bibitem{EGA2}
{\sc A. Grothendieck}, {\em \'{E}l\'ements de g\'eom\'etrie alg\'ebrique. {II}.
  {L}e langage des sch\'emas}, Inst. Hautes \'Etudes Sci. Publ. Math. {\bf  4}
  (1960).

\bibitem{groth61}
\leavevmode\vrule height 2pt depth -1.6pt width 23pt, {\em Techniques de
  construction et th\'eor\`emes d'existence en g\'eom\'etrie alg\'ebrique.
  $\hbox{IV}$. {L}es sch\'emas de {H}ilbert}, S\'eminaire Bourbaki {\bf 6},
  {E}xp. n. 221 (1960/61), 249--276.

\bibitem{henni}
{\sc A. A. Henni}, {\em Monads for framed torsion-free sheaves on
  multi-blow-ups of the projective plane}, Int. J. Math. {\bf  25} (2014),  
1450008.
\newblock 42 pages.

\bibitem{king}
{\sc A. King}, {\em Moduli of representations of finite-dimensional algebras}, Quart. J. Math. Oxford (2), {\bf 45} (1994),
515- 530.

\bibitem{kuz}
{\sc A. Kuznetsov}, {\em Quiver varieties and Hilbert schemes}, Moscow Math. J. {\bf 7} (2007), 673--697.


\bibitem{lehn2004}
{\sc M. Lehn}, {\em Lectures notes on {H}ilbert schemes}, in {\em Algebraic
  structures and moduli spaces}, C{M}{R} {P}roc. {L}ecture {N}otes, vol. 38,
  American Mathematical Society, Providence, RI, 2004,  \ 1--30.
  
  \bibitem{Mansch} {\sc J. Manschot}, {\em BPS invariants of $\mathcal N\!\!=\!4$
gauge theory on Hirzebruch surfaces}, Commun.  Number Theory and Physics {\bf 6} (2012), 497--516.

\bibitem{Mum}
{\sc D. Mumford, J. Fogarty, and F. Kirwan}, {\em Geometric invariant theory}, third enlarged ed.,
  no. 34 in Ergebnisse der Mathematik und ihrer Grenzgebiete, Springer-Verlag,
  Berlin -- Heidelberg, 1994.

\bibitem{naka1994}
{\sc H. Nakajima}, {\em Instantons on ALE spaces, quiver varieties and Kac-Moody algebras},
Duke Math. J. {\bf 76} (1994), 365--416.

\bibitem{naka1998}
\leavevmode\vrule height 2pt depth -1.6pt width 23pt, {\em Quiver varieties and Kac-Moody algebras},
Duke Math. J. {\bf 91} (1998), 515--560.

\bibitem{Nakabook}
\leavevmode\vrule height 2pt depth -1.6pt width 23pt, {\em Lectures on
  {H}ilbert schemes of points on surfaces}, University Lecture Series, vol. 18,
  American Mathematical Society, Providence, RI, 1999.
  
\bibitem{naka2014}
\leavevmode\vrule height 2pt depth -1.6pt width 23pt, {\em More lectures on {H}ilbert schemes of points on
  surfaces}, Adv. Stud. Pure Math. {\bf 69} (2016), Development of Moduli Theory -- Kyoto 2013, 173--205.
  
\bibitem{nakasaw} \leavevmode\vrule height 2pt depth -1.6pt width 23pt, 
   {\em  Cherkis bow varieties and Coulomb branches of quiver gauge theories of affine type $A$,}
  {\tt  arXiv:\allowbreak 1606.02002}.
  
 \bibitem{negut}
{\sc A. Negut,} {\em Quantum toroidal and shuffle algebras, R-matrices and a conjecture of    Kuznetsov,}
{\tt arXiv:\allowbreak 1302.6202}.

  
\bibitem{N2015} \leavevmode\vrule height 2pt depth -1.6pt width 23pt,  {\em Quantum algebras and cyclic quiver varieties,} Ph.D. Thesis,  Columbia University, 2015.

\bibitem{Ok} {\sc C. Okonek,  M. Schneider and H. Spindler,}
{\em Vector bundles on complex projective spaces},
Progress in Mathematics Vol. 3, Birkh\"auser, Boston, Mass., 1980.

\bibitem{OSV}
{\sc H. Ooguri, A. Strominger, and C. Vafa}, {\em Black hole attractors and the
  topological string}, Phys. Rev. D {\bf  70} (2004) 106007.
    
\bibitem{SV2013-I} {\sc O. Schiffmann and E. Vasserot,} {\em  The elliptic Hall algebra and the K-theory of the Hilbert scheme of $\mathbb A^2$,} Duke
Math. J. {\bf 162} (2013) ,   279--366.

\bibitem{SV2013-II}
\leavevmode\vrule height 2pt depth -1.6pt width 23pt,
{\em Cherednik algebras, W-algebras and the equivariant cohomology of the moduli space of instantons on $\mathbb A^2$,} Publ. Math. Inst. Hautes \'Etudes Sci. {\bf 118} (2013),   213--342.

  \bibitem{Sz} {\sc R. J. Szabo}, {\em Instantons, topological strings, and enumerative geometry,}
  Adv. Math. Phys.  (2010), Art. ID   107857 (70 pages).
  
  \bibitem{VDB} {\sc M. Van den Bergh}, {\em Double poisson algebras}, Trans. Amer. Math. Soc. {\bf 360} (2008), 5711--5769.

  
\bibitem{W2011} {\sc  M. Wemyss,} {\em Reconstruction algebras of type A,} 
Trans. Amer. Math. Soc. {\bf 363}
(2011),   3101--3132. 

\end{thebibliography}
\end{document}